\documentclass[10pt,a4paper,twocolumn]{article}

\usepackage[english]{babel}
\usepackage[T1]{fontenc}
\usepackage{bbm}
\usepackage{graphicx}
\usepackage{verbatim}
\usepackage{amsmath}
\usepackage{amssymb}
\usepackage{amsthm}
  \theoremstyle{plain}
  \newtheorem{theorem}{Theorem}
  \newtheorem{lemma}{Lemma}
  \newtheorem{proposition}{Proposition}
  
  \newtheorem{assumption}{Assumption}
  \newtheorem{definition}{Definition}
  
  \newtheorem{remark}{Remark}
\usepackage{mathrsfs}
\usepackage{mathtools}
\usepackage[left=1.5cm,right=1.5cm,top=2cm,bottom=2cm]{geometry}
\usepackage[colorlinks=true, allcolors=blue, hypertexnames=false]{hyperref}

\usepackage{algpseudocode,float,algorithmicx,algorithm}
\makeatletter
\newenvironment{breakablealgorithm}
{
	\begin{center}
		\refstepcounter{algorithm}
		\hrule height.8pt depth0pt \kern2pt
		\renewcommand{\caption}[2][\relax]{
			{\raggedright\textbf{\ALG@name~\thealgorithm} ##2\par}%
			\ifx\relax##1\relax 
			\addcontentsline{loa}{algorithm}{\protect\numberline{\thealgorithm}##2}%
			\else 
			\addcontentsline{loa}{algorithm}{\protect\numberline{\thealgorithm}##1}%
			\fi
			\kern2pt\hrule\kern2pt
		}
	}{
		\kern2pt\hrule\relax
	\end{center}
}
\makeatother

 \allowdisplaybreaks[4]

\begin{document}

\title{\textbf{Two-Timescale Optimization Framework for Sparse-Feedback Linear-Quadratic Optimal Control}}
\author{Lechen Feng, Yuan-Hua Ni,  Xuebo Zhang\\
\textit{\small College of Artificial Intelligence, Nankai University, Tianjin 300350, P.R. China.}\thanks{Email: fenglechen0326@163.com (Lechen Feng), yhni@nankai.edu.cn (Yuan-Hua Ni), zhangxuebo@nankai.edu.cn (Xuebo Zhang).}
}

\date{}

\maketitle

\begin{abstract}
A $\mathcal{H}_2$-guaranteed sparse-feedback linear-quadratic (LQ) optimal control with convex parameterization and convex-bounded uncertainty is studied in this paper,  where $\ell_0$-penalty is added into the $\mathcal{H}_2$ cost to penalize the number of communication links among distributed controllers.
Then, the sparse-feedback gain is investigated to minimize the modified $\mathcal{H}_2$ cost together with the stability guarantee, and the corresponding main results are of three parts.
First, the $\ell_1$ relaxation sparse-feedback LQ problem is of concern, and a two-timescale algorithm is developed based on proximal coordinate descent and primal-dual splitting approach.
Second, piecewise quadratic relaxation sparse-feedback LQ control is investigated, which exhibits an accelerated convergence rate.
Third, sparse-feedback LQ problem with $\ell_0$-penalty is directly studied through BSUM (Block Successive Upper-bound Minimization) framework, and precise approximation method and variational properties are introduced.

\end{abstract}

\section{Introduction}
During the last few decades, the concept of distributed control has gained significant attention due to its appealing characteristics, such as the potential to lower communication costs and parallelizable processing.
In particular, distributed linear-quadratic (LQ) optimal control has been also a topic of general interest, which has been extensively investigated \cite{kk3,kk4}.
Nevertheless, distributed LQ problems present a significant challenge, as obtaining a closed-form solution is exceedingly difficult, with the exception of highly specific cases \cite{Tsitsiklis-1985}.
Optimization,  an ``end-to-end'' method that allows us to optimize the performance criterion directly, has significantly contributed to the centralized LQ problems \cite{i1} and distributed LQ problems \cite{i9,i10,i11,i12}, respectively.
However, the question of how to design a controller with a general sparsity structure has yet to be resolved.
{Here, the term ``sparsity structure'' in this context refers to the presence of zero elements within the  feedback gain matrix.}
{In this paper, the $\mathcal{H}_2$-guaranteed cost sparse-feedback control problem under convex parameterization and convex-bounded uncertainty is studied, which is intrinsically linked to distributed control problem.}
Motivated by recent achievements in optimization-based controls, the objective of this study is to establish an optimization framework for sparse-feedback LQ control, allowing us to obtain a sparse-feedback gain which admits the advantages of high reliability and fault tolerance, reduced communication load and great scalability.

\subsection{Related Works}

\noindent
\textbf{1) Optimization-based Distributed LQ control.}
Sparse-feedback LQ control exhibits a strong connection with the distributed LQ problem, which has been widely considered in the existing literature.
A classic distributed LQ optimal control  can be cast as a centralized problem where the
stabilizing feedback gain matrix is restricted to lie in a particular subspace $\mathcal{K}$. If $\mathcal{K}$
 denotes the block-diagonal matrix space, classic convex parameterization methods can be utilized to obtain an
  approximate convex problem \cite{i4}, and ADMM (Alternating Direction Method of Multipliers) based method is
   shown to converge to the global optimizer of the approximate convex problem \cite{i9,i12,b6}.
    If $\mathcal{K}$ possesses general sparsity structure, the SDP (Semidefinite Programming) relaxation
     perspective is of concern. If the SDP relaxation has a rank-$1$ solution, a globally linear optimal
      distributed controller can be recovered from the solution of SDP \cite{i10}.
      {
      Moreover, an exact and inexact convex reformulation of sparse linear controller synthesis is introduced based on Youla parameterization and
      quadratically invariant assumption \cite{furieri2020sparsity,rotkowitz2005characterization}, and infinite-dimensional convex optimization problems are proposed, which can be solved approximately by Laurent expansion.
      }
      Additionally, zero-order policy optimization algorithm is utilized to optimize the distributed LQ cost directly, which can be used to find a stationary point with a PAC (Probably Approximately Correct) guarantee \cite{i11}.
\\[0.5em]
\textbf{2) Sparse optimization.} Sparse optimization, a task that is intrinsically NP-hard, is a fundamental challenge in the fields of optimization, machine learning, image processing, and control theory.
One popular method is the relaxed approximation method, for example, convex relaxed method \cite{i15}, nonconvex relaxed method \cite{i18} and DC (difference of convex functions) relaxed method \cite{i20}.
Due to an enhanced comprehension of variational characteristics associated with the $\ell_0$-norm, numerous studies have addressed sparse optimization by directly optimizing the $\ell_0$-norm \cite{i23}. Additionally, successive coordinatewise convex problem can approximate $\ell_0$-norm penalty problem \cite{ref4}.
\\[0.5em]
\textbf{3) Separable constrained optimization. }The most classic method to solve separable constrained problem is the ADMM method \cite{i24}, which is based on the  primal-dual framework.
Additionally, there exists a wide range of variations of ADMM; see \cite{i26} for details.
Currently, some dynamical system based methods are proposed to solve linear equality constrained optimization problem, for instance, first-order dynamical system \cite{i30}, second-order dynamical system \cite{i31}. In particular, \cite{ref2} introduces a new primal-dual splitting algorithms with provable nonergodic convergence rates based on differential equation solver approach.

\subsection{Contributions}
This paper examines sparse-feedback LQ optimal control.
If compared with existing results, the main contributions of the presented paper are as follows.

{\textbf{1)}} %
{
The sparse-feedback LQ problem can be represented by the global minimizer of a nonconvex optimization problem with $\ell_0$-penalty, and the convex relaxation problems are of concern.
Compared with \cite{ref1,i12,i9}, the feedback gain with general sparsity structure is investigated rather than the block-diagonal feedback gain, and instead of the infinite-dimensional optimization that is solved by finite-dimensional approximation method \cite{furieri2020sparsity,rotkowitz2005characterization,matni2016regularization}, the proposed optimization problems
of this paper is finite-dimensional.
Furthermore, in contrast to \cite{i10,kk4,arastoo2016closed,i11}, the algorithms introduced in this paper provide an optimality and stabilizing guarantee rather than heuristic algorithms or stationary-point-guarantee.
}

{\textbf{2)}}  %
{
The results of this paper is of independent interest in optimization field.
A novel algorithm framework is proposed for solving sparse optimization problem with $\ell_0$-penalty, which consists of solving coordinatewise convex function successively and is known as a precise approximation method.
In addition, the convex relaxations of $\ell_0$-norm are considered, and efficient algorithms are proposed for the aforementioned convex relaxation problems.
Although existing algorithmic frameworks are applied to solve the convex relaxation problems, these frameworks involve solving subproblems, requiring additional effort of designing suitable iterative algorithm for subproblems through considering the unique structure of LQ control.
}
Furthermore, we have established a direct relationship between the acceleration of solving sparse-feedback LQ problems and the design of strongly convex relaxation problems. To our best knowledge, this phenomenon has not been reported in the field of sparse optimization.

{\textbf{3)}}  The variational properties of $\ell_0$-penalty is firstly studied for sparse-feedback LQ problem.
Motivated by the convex relaxation framework of \cite{ref4}, we show that the $\ell_1$ convex relaxation problems can be seen as approximations of $\ell_0$-penalty sparse-feedback LQ problem.
Noting that \cite{ref4} studies an unconstrained smooth optimization problem, the considered problem of this paper is constrained and nonsmooth. Moreover, the results in \cite{ref4} rely on the assumption that one can obtain a global minimizer of a nonconvex nonsmooth optimization by iterative shrinkage-thresholding algorithms (ISTA), whereas such a strong assumption is not required and satisfied in this paper.  {Hence, the conclusion and techniques of proof presented in our analysis do not exhibit parallel generalizations of those of \cite{ref4}.}

{\textbf{4)}} The sparse-feedback LQ problem studied in this paper is closely related to the distributed LQ control with fixed communication topology, and hence the proposed algorithms and methodologies also make contributions to the field of distributed LQ controls. In addition, the studied sparse-feedback LQ problem also covers the special setting: we aim to obtain a feedback gain as sparse as possible while maintaining the communication topology constraint; to the best of authors' knowledge, such a special distributed LQ formulation has not been reported yet in the existing literature.

{\textit{Notation}}. Let $\Vert \cdot\Vert$, $\Vert\cdot\Vert_F$, $\Vert\cdot\Vert_0$ and $\Vert \cdot\Vert_{\mathcal{H}_2}$ be the spectral norm, Frobenius norm, $\ell_0$-norm, and $\mathcal{H}_2$-norm of a matrix respectively. $\mathbb{S}^n$ is the set of symmetric matrices of $n\times n$;  $\mathbb{S}^n_{++}$ ($\mathbb{S}^n_+$)  is the set of positive (semi-)definite matrices of $n\times n$; $I_n$ is identity matrix of $n\times n$; $\mathbf{0}_n$ is $(0,\dots,0)^\top\in\mathbb{R}^n$.  $A\succ B(A\succeq B)$ means that the matrix $A-B$ is positive (semi-)definite and $A\prec B(A\preceq B)$ means that the matrix $A-B$ is negative (semi-)definite.
Given $A\in\mathbb{R}^{m\times n}$, $A^\dagger$ denotes the Moore-Penrose generalized inverse of $A$.
{
For $A\in\mathbb{R}^{m\times n}=(a_{ij})$, $\mathrm{vec}(A)=(a_{11},\dots, a_{mn})^\top\in\mathbb{R}^{mn}$.
The operator $\langle A,B\rangle$ denotes the Frobenius inner product, i.e., $\langle A,B\rangle=\mathrm{Tr}(A^\top B)$ for all $A,B\in\mathbb{R}^{m\times n}$, and the notation $\otimes$ denotes the Kronecker product of two matrices.
For $\tau>0$, introduce the proximal operator of $g$
\begin{equation*}
  \mathbf{prox}_{\tau g}^{\mathcal{Y}}(z):=\underset{y\in\mathcal{Y}}{\operatorname*{argmin}}\left\{g(y)+\frac{1}{2\tau}\left\|y-z\right\|^2\right\}.
\end{equation*}
}\noindent
For any $\alpha\in\mathbb{R}$, the $\alpha$-level set of a function $f:\mathbb{R}^n\to[-\infty,\infty]$ is the set
${\rm lev}_{\leq\alpha}f=\{x\in\mathbb{R}^n\colon f(x)\leq\alpha\}.$
We introduce the notation $M\lesssim N$, meaning $M\leq CN$ with some generic bounded constant $C>0$.
{
For any subset $C\subseteq\mathbb{R}^n$, $\delta_{C}(\cdot)$ is the indicator function of $C$.

{
The rest of this paper is organized as follows.
In Section \ref{section2}, the basic formulations and assumptions of sparse-feedback LQ problem are proposed.
In Section \ref{section3}, we will study the $\ell_1$ relaxation problem of sparse-feedback LQ control and design a two-timescale algorithm framework to solve this convex relaxation problem.
In addition, in Section \ref{acc_section}, a piecewise quadratic convex relaxation problem is of concern, which admits an accelerated convergence rate and reveals the novel phenomenon of the link between strongly convex relaxation and acceleration of solving sparse-feedback LQ control.
In Section \ref{section6}, we directly tackle the challenge of $\ell_0$-penalty sparse-feedback LQ problem which is nonconvex, nonsmooth and exhibits a significant difficulty.
Finally, in Section \ref{section_link}, we discuss the relation between the sparse-feedback LQ problem studied in this paper and the distributed LQ problem which has been widely studied in existing papers.
}

\section{Sparse-Feedback LQ Formulation}\label{section2}
Consider a linear time-invariant system
\begin{equation}\label{system}
\begin{aligned}
\dot{x}(t)&=Ax(t)+B_2u(t)+B_1w(t),\\
z(t)&=Cx(t)+Du(t)
\end{aligned}
\end{equation}
with state $x(t)\in\mathbb{R}^n$, input $u(t)\in\mathbb{R}^m$, exogenous disturbance input $w(t)\in\mathbb{R}^l$, controlled output $z(t)\in\mathbb{R}^q$ and matrices $A,B_2,B_1,C,D$ with proper sizes. The infinite-horizon LQ problem is to find a linear static state feedback gain $K$ such that
$u(t)=-Kx(t)$
minimizes the following performance criterion
\begin{equation}\label{J}
J=\int_{0}^{\infty}z(t)^\top z(t){\rm d}t.
\end{equation}
By \cite{b1}, to optimize (\ref{J}), it is equivalent to minimize the $\mathcal{H}_2$-norm of the transfer function
\begin{equation*}
H(s)=(C-DK)(sI_n-A+B_2K)^{-1}B_1
\end{equation*}
from $w$ to $z$, and the performance criterion can be reformulated as
\begin{equation*}
  J(K)=\Vert H(s)\Vert_{\mathcal{H}_2}^2={\rm Tr}\left((C-DK)W_c(C-DK)^\top\right),
\end{equation*}
where $W_c$ is the controllability Gramian associated with the closed-loop system.
In this paper, the stabilizing controller below refers to stabilizing system (\ref{system}) internally; see Section 14 of \cite{Zhou-1996} for details.
We assume the following hypotheses hold, which has also been introduced in \cite{ref1,i11,i14}.
\begin{assumption}\label{ass1}
Assume that $C^\top D=0,D^\top D\succ 0,B_1B_1^\top\succ0, (A,B_2)$ is stabilizable and $(A,C)$ has no unobservable modes on the imaginary axis.
\end{assumption}

\begin{remark}\label{ass_1_explain}
	If $z(t)$ exhibits the following form that is studied in \cite{i9}
\begin{equation*}
  z(t)=\begin{bmatrix}
         Q^{\frac{1}{2}} \\
         0
       \end{bmatrix}x(t)+
       \begin{bmatrix}
         0 \\
         R^{\frac{1}{2}}
       \end{bmatrix}u(t)
\end{equation*}
with $Q,R\succ 0$, then $C^\top C\succ0$, $D^\top D\succ 0$ and $C^\top D=0$ naturally hold.
\end{remark}
{
In this paper, the sparse-feedback LQ problem is of concern, which can be formulated by the following optimization problem
\begin{equation}\label{sparse_feedback_1}
	\begin{aligned}
            \min_{K\in\mathbb{R}^{m\times n}}~~& J(K)+\gamma\Vert K\Vert_0\\
            {\rm s.t.}~~&K\in\mathcal{S}
    \end{aligned}
\end{equation}
with stabilizing feedback gain set $\mathcal{S}$ and constant $\gamma\geq 0$. Remarkably, problem \eqref{sparse_feedback_1} may degenerate to the classic infinite-horizon LQ problem when $\gamma=0$, which has been widely investigated in the existing literature, for example \cite{Zhou-1996}. However, optimization problem \eqref{sparse_feedback_1} admits certain significant difficulties: nonconvex LQ cost $J(K)$, nonconvex stabilizing feedback set $\mathcal{S}$ and discontinuous term $\Vert K\Vert_0$;
and thus, through introducing the augmented variables, a classic parameterization strategy \cite{Zhou-1996}} is utilized to simplify problem \eqref{sparse_feedback_1}.
}
Concretely, we introduce the following notations {with $p=m+n$}
\begin{equation*}
  \begin{aligned}
  &F=\begin{bmatrix}
      A & B_2 \\
      0 & 0
    \end{bmatrix}\in\mathbb{R}^{p\times p},~
    G=\begin{bmatrix}
        0  \\
        I_m
      \end{bmatrix}\in\mathbb{R}^{p\times m},\\
  & Q=\begin{bmatrix}
        B_1B_1^\top & 0 \\
        0 & 0
      \end{bmatrix}\in\mathbb{S}^p,~
      R=\begin{bmatrix}
          C^\top C & 0 \\
          0 & D^\top D
        \end{bmatrix}\in\mathbb{S}^p.
  \end{aligned}
\end{equation*}

\begin{assumption}\label{ass2}
The parameter $F$ is unknown but convex-bounded, i.e., $F$ belongs to a polyhedral domain, which is expressed as a convex combination of the extreme matrices, where $F=\sum_{i=1}^{M}\xi_iF_i,\xi_i\geq 0,\sum_{i=1}^{M}\xi_i=1$, and $F_i=\begin{bmatrix}
                              A_i & B_{2,i} \\
                              0 & 0
                            \end{bmatrix}\in\mathbb{R}^{p\times p}$
denotes the extreme vertex of the uncertain domain.
\end{assumption}

Let block matrix
\begin{equation*}
W=\begin{bmatrix}
W_1 & W_2\\
W_2^\top & W_3
\end{bmatrix}\in\mathbb{S}^p
\end{equation*}
with $W_1\in\mathbb{S}_{++}^n,W_2\in\mathbb{R}^{n\times m},W_3\in\mathbb{S}^m$. Regard $\Theta_i(W)=F_iW+WF_i^\top+Q$ as a block matrix, i.e.,
\begin{equation*}
\Theta_i(W)=\begin{bmatrix}
\Theta_{i,1}(W) & \Theta_{i,2}(W)\\
\Theta_{i,2}^\top(W) & \Theta_{i,3}(W)
\end{bmatrix}\in\mathbb{S}^p.
\end{equation*}

The following theorem introduces a subset of stabilizing controller gains, and finds an upper bound of performance criterion (\ref{J}).
\begin{theorem}\label{thm_para}(\cite{ref1})
  One can define the set
  \begin{align*}
  \mathscr{C}&=\{W\in\mathbb{S}^p:W\succeq0,\Theta_{i,1}(W)\preceq0,\forall i=1,2,\ldots,M\},
  \end{align*}
and let
    $\mathscr{K}=\{K=W_2^\top W_1^{-1}\colon W\in\mathscr{C}\}$.
  Then,
  \begin{enumerate}
    \item $K\in\mathscr{K}$ stabilizes the closed-loop system;
    \item $K\in\mathscr{K}$ gives
    $$\langle R,W\rangle\geq \Vert H_i(s)\Vert_{\mathcal{H}_2}^2,~~i=1,\dots,M,$$
    where $\Vert H_i(s)\Vert_{\mathcal{H}_2}$ represents the $\mathcal{H}_2$-norm with respect to the $i$-th extreme system.
  \end{enumerate}
\end{theorem}

\begin{remark}\label{explain_ass12}
When $F$ is known, the classic infinite-horizon LQ problem
is equivalent to the following convex optimization problem \cite{Zhou-1996}
\begin{equation}\label{classical_convex_parameterization}
         \begin{aligned}
           \min_{W\in\mathbb{S}^p}~~&\langle R,W\rangle\\
           {\rm s.t.}~~&W\in\mathscr{C}.
         \end{aligned}
\end{equation}
However, we want to expand the range of applicability in cases where $F$ is unknown but convex-bounded; see Assumption \ref{ass2}. Remarkably, LQ system with unknown but convex-bounded parameters is widely considered in the field of Linear Parameter Varying (LPV) control, which has many applications in the industrial sector; see Section 10.2.3 of \cite{sename2013robust} for details. In this case, by Theorem \ref{thm_para}, the optimization problem \eqref{classical_convex_parameterization} is no longer equivalent to classic LQ problem, since the objective function of \eqref{classical_convex_parameterization} is the upper bound of $J(K)$.
\end{remark}

{
By Theorem \ref{thm_para} and \cite{Zhou-1996}, we turn to study the following optimization problem
\begin{equation}\label{sparse_feedback_2}
         \begin{aligned}
           \min\limits_{W\in\mathbb{S}^p \atop K\in\mathbb{R}^{m\times n}}~~&\langle R,W\rangle+\gamma \Vert K\Vert_0\\
           {\rm s.t.}~~&W\in\mathscr{C},\\
           &K=W_2^\top W_1^{-1}.
         \end{aligned}
\end{equation}
Due to the nonlinear manifold constraint $K=W_2^\top W_1^{-1}$, optimization problem \eqref{sparse_feedback_2} exhibits substantial challenges. For any given binary matrix $X=(X_{ij})\in\{0,1\}^{m\times n}$ and denoting
\begin{equation*}
\begin{aligned}
  \mathbf{Sparse}(X)&=\{Y=(Y_{ij})\in\mathbb{R}^{m\times n}\colon \text{ if }X_{ij}=0\\
   &\text{ then } Y_{ij}=0, i=1,\dots,m;~j=1,\dots,n\},
\end{aligned}
\end{equation*}
it holds that
\begin{equation*}
  \begin{aligned}
&W_2^\top \in \mathbf{Sparse}(X) \text{ and } W_1 \in \mathbf{Sparse}(I_n) \\
&\quad\quad\quad\quad\quad\quad\quad\quad\Downarrow \\
&W_2^\top W_1^{-1} \in \mathbf{Sparse}(X).
\end{aligned}
\end{equation*}
Hence, it suffices to show that feedback gain $K=W_2^\top W_1^{-1}$ shares the similar sparsity structure with $W_2^\top$ provided that $W_1\in \mathbf{Sparse}(I_n)$; thus minimizing $\Vert K\Vert_0$ can be relaxed to minimize $\Vert W_2^\top\Vert_0$ with additional linear constraint that $W_1\in \mathbf{Sparse}(I_n)$. By default, the relaxation of optimization problem \eqref{sparse_feedback_2} is established as follows
\begin{equation}\label{sparse_feedback_3}
         \begin{aligned}
           \min\limits_{W\in\mathbb{S}^p}~~&\langle R,W\rangle+\gamma \Vert W_2^\top\Vert_0\\
           {\rm s.t.}~~&W\in\mathscr{C},\\
           &W_1\in \mathbf{Sparse}(I_n).
         \end{aligned}
\end{equation}
}
{\noindent
Letting $N=n(n-1)/2$, $e_i=(\mathbf{0}_{i-1}^\top,1,\mathbf{0}_{p-i}^\top)^\top\in\mathbb{R}^p$,
    \begin{equation*}
           \begin{aligned}
           &E_i^j=\begin{bmatrix}
                   e_i^\top \\
                   \vdots \\
                   e_i^\top
                 \end{bmatrix}\in\mathbb{R}^{j\times p},~~
           \hat{E}_i^j=\begin{bmatrix}
                          e_i \\
                          \vdots \\
                          e_j
                        \end{bmatrix}\in\mathbb{R}^{p(j-i+1)},\\
           &\begin{bmatrix}
                     V_{11} \\
                     \vdots \\
                     V_{N1}
                   \end{bmatrix}=\begin{bmatrix}
                    E_{1}^{n-1} \\
                    E_2^{n-2} \\
                    \vdots \\
                    E_{n-1}^1
                  \end{bmatrix}\in\mathbb{R}^{N\times p},~~
           \begin{bmatrix}
                    V_{12} \\
                    \vdots \\
                    V_{N2}
                  \end{bmatrix}=\begin{bmatrix}
                    \hat{E}_2^n \\
                    \hat{E}_3^n \\
                    \vdots \\
                    \hat{E}_n^n
                  \end{bmatrix}\in\mathbb{R}^{Np},
           \end{aligned}
 \end{equation*}
}\noindent
then problem \eqref{sparse_feedback_3} can be equivalently converted to
\begin{equation}\label{sparse_feedback_4}
\begin{aligned}
\min_{W,P}~~&\langle R,W\rangle+\gamma \Vert P\Vert_0\\
{\rm s.t.}~~&W\in\mathbb{S}^p_+,\\
&\Psi_i\in\mathbb{S}_+^{n},\forall i=1,\dots,M,\\
&-V_{j1}WV_{j2}=0,\forall j=1,\dots,N,\\
&V_1WV_2^\top-P=0
\end{aligned}
\end{equation}
with $\gamma>0$, $V_1=[0,I_m],V_2=[I_n,0]$ and
$$\Psi_i=-V_2(F_iW+WF_i^\top+Q)V_2^\top,\forall i=1,\dots,M.$$

\begin{remark}\label{group_sparse}
The group-sparse-feedback LQ problem can be formulated from the optimization viewpoint likewise.
Specifically, by replacing $\Vert P\Vert_0$ of problem \eqref{sparse_feedback_4} by $\Vert P\Vert_{\rm gs}$ (see \cite{beck2019optimization} for the definition of group-$\ell_0$ norm of $P$) and modifying the constraint such that $W_1$ is a block-diagonal matrix with proper size, a group-sparse-feedback $K=W_2^\top W_1^{-1}$ could be obtained.
Remarkably, all the results of this paper could be generalized into the group-sparse setting, and for the sake of simplicity, we only focus on how to obtain a sparse (instead of group sparse) feedback gain in the following sections.
\end{remark}

\section{$\ell_1$ Relaxation Sparse-Feedback LQ Problem}\label{section3}

The $\ell_1$-norm of a matrix $P$ exhibits the following form
\begin{equation*}
  g(P)=\sum_{i,j} w_{ij}\vert P_{ij}\vert
\end{equation*}
with $w_{ij}>0$ for all $i,j$, which may serve as a convex relaxation of $\Vert P\Vert_0$.
In this section, the $\ell_1$ convex relaxation of problem \eqref{sparse_feedback_4} is of concern.
By replacing $\Vert P\Vert_0$ by $g(P)$ above, $\ell_1$ convex relaxation sparse-feedback LQ problem \eqref{sparse_feedback_4} can be expressed in the form
\begin{equation}\label{opt2}
  \begin{aligned}
  \min_{W,P}~~&\langle R,W\rangle+\gamma g(P)\\
  {\rm s.t.}~~&\mathcal{G}(W)\in\mathcal{K},\\
  &V_1WV_2^\top-P=0
  \end{aligned}
\end{equation}
with
\begin{equation*}
  \begin{aligned}
\mathcal{G}(W)&=(W,\Psi_1,\dots,\Psi_M,V_{11}WV_{12},\dots,V_{N1}WV_{N2}),\\
\mathcal{K}&=\mathbb{S}_+^p\times\mathbb{S}_+^n\times\cdots\times\mathbb{S}_+^n\times 0\times\cdots\times 0.
  \end{aligned}
\end{equation*}

\begin{assumption}\label{slater}
  Problem (\ref{opt2}) is strictly feasible, i.e., Slater's condition is satisfied and strong duality holds.
\end{assumption}

\begin{remark}
This assumption is popular in solving distributed LQ problems via optimization landscape \cite{ref1,i12,i9}. Though the question when Assumption \ref{slater} is valid remains unresolved, and \cite{i9} offers some insights from the standpoint of block-diagonal strategy when $(A,B_2)$ is known. Specifically, let $D$ be the feasible set of optimization problem (\ref{opt2}), and denote
$$\widetilde{D}=\{W_2\in\mathbb{R}^{n\times m}\colon W_2=\mathrm{blocdiag}\{W_{2,1},\dots,W_{2,p}\}\}$$
with $W_{2,j}\in\mathbb{R}^{n_j\times m_j}$ for all $j=1,\dots, p$ and $\sum_{j=1}^p m_j=m$, $\sum_{j=1}^p n_j=n$. It is demonstrated that, if system (\ref{system}) is fully actuated or weakly coupled in terms of topological connections or weakly coupled in terms of dynamical interactions (see Definition 4, Definition 5 and Definition 8 of \cite{i9}), then $D\cap\widetilde{D}$ (a subset of $D$) is nonempty (see Proposition 1-3 of \cite{i9}). Hence, if system (\ref{system}) is fully actuated or weakly coupled in terms of topological connections or weakly coupled in terms of dynamical interactions, problem (\ref{opt2}) is feasible. Combining with the fact that $B_1B_1^\top\succ 0$ and $W_3$ is a free variable to be chosen  such that $W\succ 0$ (utilizing Schur complement lemma), problem (\ref{opt2}) is strictly feasible, i.e., Assumption \ref{slater} holds.
It is important to note that the feasible set of (\ref{opt2}) is larger than the ones in \cite{ref1,i12,i14,i9}, since there is no structural constraint on $W_2$ in this paper.
\end{remark}

Letting $\widetilde{W}={\rm vec}(W),\widetilde{P}={\rm vec}(P)$ and denoting $\mathcal{A},\mathcal{B},\widehat{\mathcal{A}}$ with
\begin{equation*}
  \mathcal{A}=\begin{pmatrix}
  V_{12}^\top\otimes V_{11}\\
  \vdots\\
  V_{M2}^\top\otimes V_{M1}\\
  V_2\otimes V_1
  \end{pmatrix}\doteq \begin{pmatrix}
    \widehat{\mathcal{A}}\\
    V_2\otimes V_1
  \end{pmatrix},~~
  \mathcal{B}=\begin{pmatrix}
  0\\
  \vdots\\
  0\\
  -I_{mn}
  \end{pmatrix},
\end{equation*}
problem \eqref{opt2} can be rewritten in a more compact form
\begin{equation}\label{problem2}
\begin{aligned}
\min_{\widetilde{W},\widetilde{P}}~~ &f(\widetilde{W})+h(\widetilde{P})\\
{\rm s.t.}~~&\mathcal{A}\widetilde{W}+\mathcal{B}\widetilde{P}=0,
\end{aligned}
\end{equation}
where
\begin{align*}
f_1(\widetilde{W})&=\langle {\rm vec}(R),\widetilde{W}\rangle,\\
f_2(\widetilde{W})&=\delta_{\Gamma_+^p}(\widetilde{W})+\delta_{\Gamma_+^n}(\widetilde{\Psi}_1)+\dots+\delta_{\Gamma_+^n}(\widetilde{\Psi}_M),\\
f(\widetilde{W})&=f_1(\widetilde{W})+f_2(\widetilde{W}),\\
h(\widetilde{P})&=\gamma g(\widetilde{P})
\end{align*}
with $\widetilde{\Psi}_i={\rm vec}(\Psi_i)$ and
 $ \Gamma_+^p=\{ {\rm vec}(W)\colon W\in\mathbb{S}_+^p \}$.
In this section, we first briefly review the basic setup of primal-dual splitting method and BSUM framework, while then a two-timescale optimization framework is proposed to solve optimization problem \eqref{problem2}. Finally, the convergence analysis and variational properties are introduced.

\subsection{Review of Optimization Background}

The primal-dual splitting \cite{ref2} is a first-order method which solves a separable convex optimization problem of the form:
\begin{equation}\label{sep_opt}
  \min_{x\in\mathcal{X},y\in\mathcal{Y}} F(x,y)=\bar{f}(x)+\bar{g}(y)~~{\rm s.t.}~Ax+By=b,
\end{equation}
where $\mathcal{X},\mathcal{Y}$ are two closed convex sets, and $\bar{f}, \bar{g}$ are two proper closed convex functions. Assume that $\bar{g}$ is closed and $\mu_g$-strongly convex with $\mu_g>0$ and $\bar{f}=f_1+f_2$ with closed convex $f_2$,  $\mu_f$-strongly convex and $L_f$-smooth $f_1$ and $0\leq\mu_f\leq L_f<\infty$. Then, the following primal-dual splitting method is proposed to solve problem \eqref{sep_opt}
\begin{equation}\label{Luo}
\left\{
\begin{aligned}
&u_k=(x_k+\alpha_kv_k)/(1+\alpha_k),\\
&d_k=\nabla f_1(u_k)+A^\top\lambda_k,\\
&v_{k+1}=\underset{v\in\mathcal{X}}{\operatorname*{argmin}}\{f_{2}(v)+\langle d_{k},v\rangle\\
&+\frac{\alpha_{k}}{2\theta_{k}}\left\|Av+Bw_{k}-b\right\|^{2}+\frac{\widetilde{\eta}_{f,k}}{2\alpha_{k}}\|v-\widetilde{v}_{k}\|^{2}\},\\
&x_{k+1}=(x_k+\alpha_kv_{k+1})/(1+\alpha_k),\\
&\bar{\lambda}_{k+1}=\lambda_k+\alpha_k/\theta_k(Av_{k+1}+Bw_k-b),\\
&y_{k+1}=\mathbf{prox}_{\tau_{k}\bar{g}}^{\mathcal{Y}}(\widetilde{y}_{k}-\tau_{k}B^{\top}\bar{\lambda}_{k+1}),\\
&w_{k+1}=y_{k+1}+(y_{k+1}-y_{k})/\alpha_k,\\
&\lambda_{k+1}=\lambda_k+\alpha_k/\theta_k(Av_{k+1}+Bw_{k+1}-b),
\end{aligned}
\right.
\end{equation}
where $(\tau_k,\tilde{v}_k,\tilde{y}_k,\tilde{\eta}_{f,k},\eta_{g,k})$ is given by
\begin{align}
\label{tau}&\tau_{k}=\alpha_{k}^{2}/\eta_{g,k},\\
\label{f2} &\eta_{g,k}=(\alpha_k+1)\beta_k+\mu_g\alpha_k,\\
\label{f3} &\widetilde{y}_k=y_k+\frac{\alpha_k\beta_k}{\eta_{g,k}}(w_k-y_k),\\
\label{f4} &\widetilde{\eta}_{f,k}=\gamma_k+\mu_f\alpha_k,\\
\label{f5} &\widetilde{v}_k=\frac{1}{\widetilde{\eta}_{f,k}}(\gamma_kv_k+\mu_f\alpha_k u_k),
\end{align}
and $\{(u_k,d_k,v_k,x_k,\bar{\lambda}_k,y_k,w_k,\lambda_k)\}$ is the iteration sequence.
To emphasis, $\alpha_k>0$ is the step size and the parameter system is given by
\begin{equation}\label{para}
\begin{aligned}
  &\frac{\theta_{k+1}-\theta_k}{\alpha_k}=-\theta_{k+1},\\
  &\frac{\gamma_{k+1}-\gamma_k}{\alpha_k}=\mu_f-\gamma_{k+1},\\
  &\frac{\beta_{k+1}-\beta_k}{\alpha_k}=\mu_g-\beta_{k+1}
\end{aligned}
\end{equation}
with initial conditions $\theta_0=1,\gamma_0>0$ and $\beta_0>0$.

\begin{remark}\label{Luopaper-1}
In fact, the following differential inclusion
\begin{equation}\label{ode}
  \left\{
  \begin{aligned}
  &0\in \gamma x''+(\gamma+\mu_f)x'+\partial_x \mathcal{L}(x,y,\lambda),\\
  &0=\theta\lambda'-\nabla_\lambda \mathcal{L}(x+x',y+y',\lambda),\\
  &0\in\beta y''+(\beta+\mu_g)y'+\partial_y \mathcal{L}(x,y,\lambda)
  \end{aligned}
  \right.
\end{equation}
is the continuous-time version of algorithm (\ref{Luo}), where
\begin{equation}\label{Lagrangian2}
  \mathcal{L}(x,y,\lambda)=\bar{f}(x)+\bar{g}(y)+\left\langle \lambda,Ax+By-b\right\rangle.
\end{equation}
Obviously, differential inclusion (\ref{ode}) encompasses two Nesterov-type differential inclusions pertaining to primal variables $x$ and $y$, as well as a gradient flow about the dual variable $\lambda$. Besides, an alternative presentation of (\ref{ode}) reads as
\begin{equation}\label{ode2}
  \left\{
  \begin{aligned}
  &x'=v-x,\\
  &\gamma v'\in\mu_f(x-v)-\partial_x\mathcal{L}(x,y,\lambda),\\
  &\theta\lambda'=\nabla_\lambda\mathcal{L}(v,w,\lambda),\\
  &\beta w'\in\mu_g(x-v)-\partial_y\mathcal{L}(x,y,\lambda),\\
  &y'=w-y.
  \end{aligned}
  \right.
\end{equation}
Therefore, (\ref{Luo}) represents the semi-implicit Euler discretization of differential inclusion (\ref{ode2}), which justifies its classification as a differential equation solver approach.
 We employ this specific primal-dual splitting method instead of the classic ADMM primarily because it allows us to establish a connection between the acceleration of solving sparse-feedback LQ problems and the design of strongly convex relaxation problems. This relationship will be further discussed in the subsequent sections.
\end{remark}

In general, separable optimization problems exhibit favorable properties. However, in practical scenarios, the majority of optimization problems are non-separable, with optimization variables being intricately coupled. Concretely, the following non-separable (possibly nonconvex) problem is considered in \cite{ref9}:
\begin{equation}\label{separable_nonconvex}
  \min~f(x),~~{\rm s.t.~}x\in\mathcal{X},
\end{equation}
where $x=(x_1,\dots,x_n)^\top\in\mathbb{R}^m$, and $\mathcal{X}\subseteq\mathbb{R}^m$ is the Cartesian product of $n$ closed convex sets: $\mathcal{X}=\mathcal{X}_1\times\dots\times\mathcal{X}_n$ with $\mathcal{X}_i\in\mathbb{R}^{m_i}$ and $\sum_i m_i=m$.

\begin{definition}\label{auxiliary}
  For a continuous function $f:\mathcal{X}\to\mathbb{R}$, where $\mathcal{X}\subseteq\mathbb{R}^m$ is the Cartesian product of $n$ closed convex sets: $\mathcal{X}=\mathcal{X}_1\times\dots\times\mathcal{X}_n$ with $\mathcal{X}_i\in\mathbb{R}^{m_i}$ and $\sum_i m_i=m$. We call $\{u_i(x_i,x)\}_{i=1}^n$ are auxiliary functions of $f$, if
  \begin{enumerate}
    \item $u_i(y_i,y)=f(y)~~\forall y\in\mathcal{X},\forall i$,
    \item $u_i(x_i,y)\geq f(y_1,\dots,y_{i-1},x_i,y_{i+1},\dots,y_n)$ $\forall x_i\in\mathcal{X}_i$, $\forall y\in\mathcal{X},\forall i,$
    \item $u_i'(x_i,y;d_i)|_{x_i=y_i}=f'(y;d)$ $\forall d=(0,\dots,d_i,\dots,0)$, s.t. $y_i+d_i\in\mathcal{X}_i$ $\forall i$,
    \item $u_i(x_i,y)$ is continuous in $(x_i,y) $ $\forall i$.
  \end{enumerate}
\end{definition}

The following pseudocode of BSUM framework proposed in \cite{ref9} gives methodology for
solving separable optimization problem \eqref{separable_nonconvex}, where $\{u_i\}_{i=1}^n$ is a sequence of auxiliary functions of $f$.
\begin{breakablealgorithm}
  \caption{BSUM}
  \begin{algorithmic}
  \State \textbf{Given} : Feasible initial point $x^0\in\mathcal{X}$
  \State \textbf{Result}: $x^r$
  \State Set $r=0$
  \For{$k=0,1,2,...$}
      \State Let $r=r+1,i=(r\mod n)+1$
      \State Let $\mathcal{X}^r=\underset{x_i\in\mathcal{X}_i}{\operatorname*{argmin}}~u_i(x_i,x^{r-1})$
      \State Set $x^r_i$ to be an arbitrary element in $\mathcal{X}^r$
      \State Set $x_k^r=x_k^{r-1}~\forall k\neq i$
      \If{Stopping Criterion$==$True}\\
          ~~~~~~~~~~\Return $x^r$
      \EndIf

  \EndFor
  \end{algorithmic}
  \label{alg4}
\end{breakablealgorithm}

Significantly, the introduced subproblems of primal-dual splitting method \eqref{Luo} and BSUM framework (Algorithm \ref{alg4}) imply that such approaches do not fully solve the optimization problems \eqref{sep_opt} and \eqref{separable_nonconvex}. Indeed, solving the subproblem may present greater challenges than solving the original optimization problem itself, and design the suitable auxiliary functions of a general nonconvex problem may entail considerable difficulties, further complicating BSUM process. Hence, in the rest of this paper, we will discuss how to design efficient algorithms for solving the subproblems and how to construct appropriate auxiliary functions, and both of which represent important topics of independent research in the field of optimization.

\subsection{Solving Inner Layer Optimization Problems}

{
In this subsection, we will discuss how to solve the subproblems when utilizing primal-dual splitting to solve problem \eqref{problem2}. Concretely, iterative scheme \eqref{Luo} consists of two subproblems: updating $v_{k+1}$ and $y_{k+1}$. Obviously, the solution $\widetilde{P}_{k+1}$ of subproblem w.r.t. $y_{k+1}$ is given by the following proximal operator when using the symbol system of this paper
}
\begin{equation}\label{prox1}
\mathbf{prox}_{\tau_{k}h}^{\mathbb{R}^{mn\times 1}}\left(\widetilde{P}_{k}+\frac{\alpha_k\beta_k}{\eta_{g,k}}(w_k-\widetilde{P}_k)-\tau_{k}\mathcal{B}^{\top}\bar{\lambda}_{k+1}\right).
\end{equation}
Denote
\begin{equation}\label{Omega}
\Omega^k={\operatorname*{vec}}^{-1}\left(\widetilde{P}_k+\frac{\alpha_k\beta_k}{\eta_{g,k}}(w_k-\widetilde{P}_k)-\tau_k\mathcal{B}^\top\bar{\lambda}_{k+1}\right),
\end{equation}
$\rho_k=\frac{1}{\tau_k}$, $a_k=\frac{\gamma}{\rho_k}w_{ij}$, and let
\begin{equation}\label{P*}
\begin{aligned}
  P_{ij;k+1}^*&\doteq  \underset{P_{ij}\in\mathbb{R}}{\operatorname*{argmin}} \{\gamma w_{ij}\vert  P_{ij}\vert+\frac{\rho_k}{2}(P_{ij}-\Omega^k_{ij})^2\}\\
  &=\left\{
  \begin{aligned}
  &(1-{a_k}/{\vert\Omega_{ij}^k\vert})\Omega_{ij}^k,~~&\vert \Omega_{ij}^k\vert >a_k,\\
  &0,~~&\vert \Omega_{ij}^k\vert \leq a_k.
  \end{aligned}
  \right.
\end{aligned}
\end{equation}
By Example 6.19 of \cite{ref5}, we can deduce the closed-form solution of subproblem \eqref{prox1}:
$$\widetilde{P}_{k+1}=(P_{11;k+1}^*,\dots,P_{mn;k+1}^*)^\top,$$
which exhibits naturally sparsity structure.

Subproblem w.r.t. $v_{k+1}$ is much more complicated, necessitating the implement of suitable iterative algorithm, which is the main task and contribution of this subsection.
Specifically, by the update rule of $v_{k+1}$ in \eqref{Luo} and when employing the symbol system of this paper, the subproblem w.r.t. $v_{k+1}$ can be transformed to
\begin{align}\label{subproblem2}
\min_{{\rm svec}(W)}~& \langle d_{k},D_1{\rm svec}(W) \rangle+\frac{\alpha_{k}}{2\theta_{k}}\left\|\mathcal{A}D_1{\rm svec}(W)+\mathcal{B}w_{k}\right\|^{2}\notag\\
&+\frac{\widetilde{\eta}_{f,k}}{2\alpha_{k}}\left\|D_1{\rm svec}(W)-\widetilde{v}_{k}\right\|^{2}\notag\\
{\rm s.t.}~& {\rm svec}(W)\in\widetilde{\Gamma}_+^p,\notag\\
&{\rm svec}(\Psi_i)\in\widetilde{\Gamma}_+^n,~i=1,\dots,M,
\end{align}
where $\widetilde{\Gamma}_+^p=\{ {\rm svec}(W)\colon W\in\mathbb{S}_+^p\}$, $\mathrm{svec}(\cdot)$ is the vectorization operator on $\mathbb{S}^n$ defined by
\begin{equation*}
  {\rm svec}(S)=\left[s_{11},s_{12},\cdots,s_{n1},\cdots,s_{nn}\right]^\top\in\mathbb{R}^{\frac{n(n+1)}{2}}
\end{equation*}
for $S=(s_{ij})\in\mathbb{S}^n$, and there exist matrices $T_1,T_2,D_1,D_2$ such that
\begin{equation*}
  \begin{aligned}
  &{\rm svec}(W)=T_1{\rm vec}(W),~~{\rm svec}(\Psi_i)=T_2{\rm vec}(\Psi_i),\\
  &{\rm vec}(W)=D_1{\rm svec}(W),~~{\rm vec}(\Psi_i)=D_2{\rm svec}(\Psi_i).
  \end{aligned}
\end{equation*}

\begin{remark}
We partition the identity matrix $I_N$ ($N=n(n-1)/2$) as follows
\begin{equation*}
  I_N=(u_{11},u_{21},\dots,u_{n1},u_{22},\dots,u_{n2},\dots,u_{nn}).
\end{equation*}
Then, for every matrix $A\in\mathbb{S}^{n}$, the elimination matrix $L$ (which performs the transformation ${\rm svec}(A)=L{\rm vec}(A)$) can be explicitly expressed as
\begin{equation*}
  L=\sum_{i\geq j}u_{ij}({\rm vec}E_{ij})^\top=\sum_{i\geq j} (u_{ij}\otimes e_j^\top\otimes e_i^\top)
\end{equation*}
with $E_{ij}=e_ie_j^\top$. By Lemma 3.2 of \cite{b3}, $L^\top=L^\dagger$, and naturally $LL^\top=I_N$.
Inversely, there exist matrix $D_1\in\mathbb{R}^{p^2\times\frac{p(p+1)}{2}}$ and $D_2\in\mathbb{R}^{n^2\times N}$ such that
$${\rm vec}(W)=D_1{\rm svec}(W),~~{\rm vec}(\Psi_i)=D_2{\rm svec}(\Psi_i).$$
In addition, $D_1$ can be explicitly expressed as
\begin{equation*}
D_1=\sum_{i\geq j}u_{ij}{\rm vec}(T_{ij})^\top,
\end{equation*}
where $T_{ij}$ is a $p\times p$ matrix with $1$ in position $(i,j),(j,i)$ and zero elsewhere.
\end{remark}

{By Assumption 3, for all $k$, problem (\ref{subproblem2}) satisfies Slater's condition and strong duality holds. }
The primal problem of (\ref{subproblem2}) presents a level of difficulty that justifies the exploration of its dual formulation.
\begin{lemma}
The dual problem of optimization problem \eqref{problem2} exhibits an explicit form of quadratic optimization problem with semi-definite cone constraints.
\end{lemma}
\begin{proof}
Denoting
\begin{equation*}
\begin{aligned}
&\tilde{d}_k=D_1^\top d_k-X_0\\
&~~~~~~+D_1^\top\sum_{i=1}^{M}[(V\otimes VF_i)^\top+(VF_i\otimes V)^\top]T_2^\top X_i\\
&~~~\doteq\mathcal{L}_k(X_0,X_1,\dots,X_M),\\
&\sigma_1=\frac{\alpha_k}{2\theta_k},~\sigma_2=\frac{\widetilde{\eta}_{f,k}}{2\alpha_k},~\tilde{b}_k=\mathcal{B}w_k,
\end{aligned}
\end{equation*}
the Lagrangian of (\ref{subproblem2}) can be defined as
\begin{equation}\label{Lagrangian}
\begin{aligned}
&\mathcal{L}({\rm svec}(W);X)\\
=&\langle \tilde{d}_k,{\rm svec}(W)\rangle+\sigma_1\|\mathcal{A}D_1{\rm svec}(W)+\tilde{b}_k\|^{2}\\
&+\sigma_2\|D_1{\rm svec}(W)-\widetilde{v}_{k}\|^{2}\\
&+\langle X_1+\dots+X_M,T_2(V\otimes V){\rm vec}(Q)\rangle,
\end{aligned}
\end{equation}
where $X=(X_0,X_1,\dots,X_M)\in\widetilde{\Gamma}_+^p\times\widetilde{\Gamma}_+^n\times\dots\times\widetilde{\Gamma}_+^n$ is the Lagrange multiplier.
Hence, by means of tedious yet essential calculations, the Lagrange dual function $\theta(X)$ is derived by
\begin{equation}\label{dualfunction}
\begin{aligned}
  \theta(X)&=\min_{{\rm svec}(W)}\mathcal{L}({\rm svec}(W);X)\\
  &=\langle X_1+\dots+X_M,T_2(V\otimes V){\rm vec}(Q)\rangle\\
  &~~~+\min_{{\rm svec}(W)}\hat{\theta}({\rm svec}(W))
\end{aligned}
\end{equation}
with
\begin{equation*}
\begin{aligned}
\hat{\theta}({\rm svec}(W))=&\langle \tilde{d}_k,{\rm svec}(W)\rangle+\sigma_1\|\mathcal{A}D_1{\rm svec}(W)+\tilde{b}_k\|^{2}\\
&+\sigma_2\|D_1{\rm svec}(W)-\widetilde{v}_{k}\|^{2}.
\end{aligned}
\end{equation*}
Letting $\nabla\hat{\theta}({\rm svec}(W))=0$, and denoting
\begin{equation*}
\begin{aligned}
&M=2\sigma_1(\mathcal{A}D_1)^\top(\mathcal{A}D_1)+2\sigma_2D_1^\top D_1,\\
&q_k=-2\sigma_1(\mathcal{A}D_1)^\top\tilde{b}_k+2\sigma_2D_1^\top\widetilde{v}_k,
\end{aligned}
\end{equation*}
we can obtain
\begin{equation}\label{relation}
{\rm svec}(W)=M^{-1}(q_k-\tilde{d}_k).
\end{equation}
Substituting (\ref{relation}) into (\ref{dualfunction}), it holds
\begin{align*}
  \theta(X)&=\langle X_1+\dots+X_M,T_2(V\otimes V){\rm vec}(Q)\rangle\\
  &~~~+\tilde{d}_k^\top(\sigma_1(\mathcal{A}D_1 M^{-1})^\top(\mathcal{A}D_1 M^{-1})+\sigma_2M^{-2}\\
  &~~~-M^{-1})\tilde{d}_k+\langle \tilde{d}_k,M^{-1}q_k+2\sigma_1(\mathcal{A}D_1 M^{-1})^\top\xi_k^1\\
  &~~~+2\sigma_2(D_1 M^{-1})^\top\xi_k^2\rangle+\Vert\xi_k^1\Vert^2+\Vert\xi_k^2\Vert^2,
\end{align*}
where $\xi_k^1=-\mathcal{A}D_1 M^{-1}q_k-\tilde{b}_k$, $\xi_k^2=\widetilde{v}_k-D_1 M^{-1}q_k$.
Denoting
\begin{equation*}
\begin{aligned}
\Omega_k&=M^{-1}q_k+2\sigma_1(\mathcal{A}D_1 M^{-1})^\top\xi_k^1+2\sigma_2(D_1 M^{-1})^\top\xi_k^2,\\
\Psi&=\sigma_1(\mathcal{A}D_1 M^{-1})^\top(\mathcal{A}D_1 M^{-1})+\sigma_2M^{-2}-M^{-1},
\end{aligned}
\end{equation*}
the formulation of the dual problem is presented in such a manner:
  \begin{align}\label{dualproblem1}
  \max_{X=(X_0,\dots,X_M)}~~&\langle X_1+\dots+X_M,T_2(V\otimes V){\rm vec}(Q)\rangle\notag\\
  &+\langle \tilde{d}_k,\Omega_k\rangle+\tilde{d}_k^\top\Psi\tilde{d}_k\notag\\
  {\rm s.t.}~~&X_0\in\widetilde{\Gamma}_+^p,\notag\\
  &X_i\in\widetilde{\Gamma}_+^n,~i=1,\dots,M.
  \end{align}
By the definition of $\tilde{d}_k$ which is linear mapping w.r.t. $(X_0,X_1,\dots,X_M)$, dual problem \eqref{dualproblem1} is quadratic optimization problem with semi-definite cone constraints. We finish our proof here.
\end{proof}
In the following text of this subsection, the sGS (symmetric Gauss-Seidel) proximal coordinate descent (PCD) is presented to solve the dual problem (\ref{dualproblem1}). Define
\begin{equation*}
\begin{aligned}
&~~~~\Theta(X_0,X_1,\dots,X_M)=\\
&-\langle X_1+\dots+X_M,T_2(V\otimes V){\rm vec}(Q)\rangle\\
&-\langle\mathcal{L}_k(X_0,\dots,X_M),\Omega_k\rangle\\
&+\Vert \sqrt{-\Psi}\mathcal{L}_k(X_0,\dots,X_M)\Vert^2\\
&+\delta_{\widetilde{\Gamma}_+^p}(X_0)+\delta_{\widetilde{\Gamma}_+^n}(X_1)+\dots+\delta_{\widetilde{\Gamma}_+^n}(X_M),
\end{aligned}
\end{equation*}
and the dual problem (\ref{dualproblem1}) can be equivalently expressed as
\begin{equation}\label{dualproblem2}
  \min_{X=(X_0\dots,X_M)}~\Theta(X_0,X_1,\dots,X_M).
\end{equation}

\begin{lemma}
When solving problem \eqref{dualproblem2} by sGS-PCD, every subproblem admits a closed-form solution, i.e., there exists a sequence of positive linear operator $\{\mathcal{S}_i\}_{i=0}^{M}$ such that the following problem
\begin{equation*}
  \min_{X_i}\Theta(X_0^\star,\dots,X_{i-1}^\star,X_i,{X}_{i+1}^\star,\dots,{X}_M^\star)+\frac{1}{2}\left\Vert X_i-X_i^\star\right\Vert_{\mathcal{S}_i}^2
\end{equation*}
with given $(X_0^\star,X_1^\star,\dots,X_M^\star)$ admits a closed-form minimizer for $\forall i=0,\dots, M$.
\end{lemma}

\begin{proof}
For $i=1,\dots,M$, denoting
\begin{equation*}
  \begin{aligned}
  &H_i(X_0,\dots,X_{i-1},X_{i+1},\dots,X_M)=\\
  &T_2(V\otimes V){\rm vec}(Q)+ T_2[(V\otimes VF_i)+(VF_i\otimes V)]D_1\Omega_k\\
  &-2T_2[(V\otimes VF_i)+(VF_i\otimes V)]D_1(-\Psi)\\
  &~~~\mathcal{L}_k(X_0,\dots,X_{i-1},X_{i+1},\dots,X_M),\\
  &L_i=2T_2[(V\otimes VF_i)+(VF_i\otimes V)]D_1(-\Psi)D_1^\top\\
  &~~~~~~[(V\otimes VF_i)+(VF_i\otimes V)]^\top T_2^\top,
  \end{aligned}
  \end{equation*}
we can compute
\begin{equation*}
\begin{aligned}
  &\partial_{X_i}\Theta(X_0,X_1,\dots,X_M)=-H_i(X_0,\dots,X_{i-1},\\
  &X_{i+1},\dots,X_M)+L_iX_i+\partial\delta_{\widetilde{\Gamma}_+^n}(X_i).
\end{aligned}
\end{equation*}
The subproblem in terms of variable $X_i$ in the backward sGS sweep is given by
\begin{align}\label{barX}
\bar{X}_{i}^{k+1}=\underset{X_i}{\operatorname*{argmin}}~&\Theta(X_0^k,X_1^k,\dots,X_{i-1}^k,X_i,\bar{X}_{i+1}^{k+1},\dots,\notag\\
&\bar{X}_M^{k+1})+\frac{1}{2}\Vert X_i-X_i^k\Vert_{\mathcal{S}_i}^2,
\end{align}
where $\mathcal{S}_i$ is a positive linear operator given by
$\mathcal{S}_i(X)=\rho_i\mathcal{I}(X)-L_iX$, with $\rho_i=\max\{{\rm eig}(L_i)\}$
and an identity operator $\mathcal{I}$.
The  optimality condition implies that
\begin{align*}
  0&\in-H_i(X_0^k,X_1^k,\dots,X_{i-1}^k,\bar{X}_{i+1}^{k+1},\dots,\bar{X}_M^{k+1})\\
  &~~~+L_iX_i+\partial\delta_{\widetilde{\Gamma}_+^n}(X_i)+\mathcal{S}_i(X_i-X_i^k)\\
  &=\rho_iX_i+\partial\delta_{\widetilde{\Gamma}_+^n}(X_i)-H_i(X_0^k,X_1^k,\dots,X_{i-1}^k,\\
  &~~~\bar{X}_{i+1}^{k+1},\dots,\bar{X}_M^{k+1})-\rho_iX_i^k+L_iX_i^k.
\end{align*}
Denoting
\begin{equation*}
\begin{aligned}
 \Delta_i^k=&\rho_i^{-1}(H_i(X_0^k,X_1^k,\dots,X_{i-1}^k,\bar{X}_{i+1}^{k+1},\dots,\bar{X}_M^{k+1})\\
 &+\rho_iX_i^k-L_iX_i^k),
\end{aligned}
\end{equation*}
upon careful observation, it may be noted that
\begin{equation}\label{u1}
  \bar{X}_{i}^{k+1}=(I+\rho_i^{-1}\partial\delta_{\widetilde{\Gamma}_+^n})^{-1}(\Delta_i^k)=\Pi_{\widetilde{\Gamma}_+^n}(\Delta_i^k),
\end{equation}
where the second equality is justified by Theorem 2 of \cite{ref1}.
For $i=0$, we can compute
\begin{equation*}
\begin{aligned}
\partial_{X_0}\Theta(X_0,\dots,X_M)=&L_0X_0+\partial\delta_{\widetilde{\Gamma}_+^p}(X_0)\\
&-H_0(X_1,\dots,X_M)
\end{aligned}
\end{equation*}
with $L_0=-2\Psi$,
$$H_0(X_1,\dots,X_M)=-2\Psi\mathcal{L}_k(X_1,\dots,X_M)-\Omega_k.$$
Similarly, the solution $X_{0}^{k+1}$ of subproblem w.r.t.  $X_0$ in the backward sGS sweep is given by
\begin{equation*}
\underset{X_0}{\operatorname*{argmin}}~\Theta(X_0,\bar{X}_1^{k+1},\dots,\bar{X}_M^{k+1})+\frac{1}{2}\left\Vert X_0-X_0^k\right\Vert_{\mathcal{S}_0}^2,
\end{equation*}
where the positive linear operator $\mathcal{S}_0$ is given by
$\mathcal{S}_0(X)=\rho_0\mathcal{I}(X)-L_0X$ with $\rho_0=\max\{{\rm eig}(L_0)\}$.
Then, by the same token, we have
\begin{equation}\label{u2}
X_{0}^{k+1}=(I+\rho_0^{-1}\partial\delta_{\widetilde{\Gamma}_+^p})^{-1}(\Delta_0^k)=\Pi_{\widetilde{\Gamma}_+^p}(\Delta_0^k)
\end{equation}
with
$\Delta_0^k=\rho_0^{-1}(H_0(\bar{X}_1^{k+1},\dots,\bar{X}_M^{k+1})+\rho_0X_0^k-L_0X_0^k).$
For forward sGS sweep, $X_i^{k+1}$ is given by
\begin{equation}\label{u3}
 X_{i}^{k+1}=\Pi_{\widetilde{\Gamma}_+^n}(\tilde{\Delta}_i^k), ~~i=1,\dots,M,
\end{equation}
where $\tilde{\Delta}_i^k$ is defined as
\begin{equation*}
\begin{aligned}
  \tilde{\Delta}_i^k=&\alpha_i^{-1}(H_i(X_0^{k+1},X_1^{k+1},\dots,X_{i-1}^{k+1},\bar{X}_{i+1}^{k+1},\dots,\bar{X}_M^{k+1})\\
  &+\rho_i\bar{X}_i^{k+1}-L_i\bar{X}_i^{k+1}).
\end{aligned}
\end{equation*}
Together with \eqref{u1}, \eqref{u2}, \eqref{u3} and the fact that projection $\Pi_{\mathbb{S}_+^n}(\cdot)$ can be explicitly expressed by the eigenvalue decomposition (see Section 8.1.1 of \cite{boyd2004convex}), we finish the proof here.
\end{proof}

\begin{remark}\label{explain_bsum}
  Due to the proximal term $\Vert X_i-X_i^k\Vert_{\mathcal{S}_i}^2$, optimization problem \eqref{barX} admits a closed-form solution \eqref{u1}.
  In fact, if the proximal term is removed, i.e., the update of $\bar{X}_i^{k+1}$ is given by
  \begin{equation}\label{why_proximal}
    \bar{X}_{i}^{k+1}=\underset{X_i}{\operatorname*{argmin}}~\Theta(X_0^k,\dots,X_i,\dots,\bar{X}_M^{k+1}),
  \end{equation}
  then sGS-PCD still converges to the global minimizer of \eqref{dualproblem2}.
  However, optimization problem \eqref{why_proximal} has no closed-form solution, necessitating the development of an iterative algorithm for solving problem \eqref{why_proximal}, e.g., gradient descent (GD) or Nesterov GD.
  It is worth highlighting that the extra iterative algorithm for problem \eqref{why_proximal} will significantly improve the complexity of Algorithm \ref{alg1}, and
  hence  the introduced special proximal term $\Vert X_i-X_i^k\Vert_{\mathcal{S}_i}^2$ is necessary.
\end{remark}

Introducing the proximal term for solving separable optimization problem is first proposed in \cite{b4}, and the convergence analysis can be considered by the analogous methodology of \cite{b5}, which originally studies the proximal terms in variational inequalities.
Remarkably, references \cite{b4,b5} only show the fact that introducing the proximal terms has no influence on the algorithm convergence, while how to design the suitable proximal terms remains unstudied, necessitating a case-by-case analysis of individual optimization problem.
Hence, the proximal term $\Vert X_i-X_i^k\Vert_{\mathcal{S}_i}^2$ in this paper does not exhibit parallel generalizations of those of \cite{b4,b5}, but requires additional efforts with consideration of the special structure of dual function $\Theta$; see Remark \ref{explain_bsum} for details.

\begin{definition}\label{strict_aux}
Consider the following optimization problem
\begin{equation*}
\begin{aligned}
  \min~~ &f(x)=g(x_1,\dots,x_K)+\sum_{i=1}^{K} h_i(x_i)\\
  {\rm s.t.}~~&x_i\in\mathcal{X}_i,~i=1,\dots,K,
\end{aligned}
\end{equation*}
where $g$is a smooth convex function and  $h_i$ is a closed convex function. We call that $\{u_i\}$ is a sequence of strict auxiliary functions, if
\begin{enumerate}
  \item $u_i(y_i,y)=g(y)~~\forall y\in\mathcal{X}=\mathcal{X}_1\times\dots\times\mathcal{X}_K,\forall i$,
  \item $u_i(x_i,y)\geq g(y_1,\dots,y_{i-1},x_i,y_{i+1},\dots,y_n)~~\forall x_i\in\mathcal{X}_i,\forall y\in\mathcal{X},\forall i,$
  \item $\nabla u_i(y_i,y)=\nabla_i g(y),\forall y\in\mathcal{X},\forall i$,
  \item $u_i(x_i,y)$ is continuous in $(x_i,y), \forall i$. Further, for any given $y\in\mathcal{X}$, it is proper, closed and strongly convex function of $x_i$, satisfying
      \begin{equation*}
        \begin{aligned}
        u_i(x_i,y)\geq &u_i(\hat{x}_i,y)+\left\langle \nabla u_i(\hat{x}_i,y),x_i-\hat{x}_i\right\rangle\\
        &+\frac{\gamma_i}{2}\Vert x_i-\hat{x}_i\Vert^2,\forall x_i,\hat{x}_i\in\mathcal{X}_i,
        \end{aligned}
      \end{equation*}
      where $\gamma_i>0$ is independent of the choice of $y$.
  \item For any given $y\in\mathcal{X}$, $u_k(x_i,y)$ has Lipschitz continuous gradient, i.e., $\forall \hat{x}_i,x_i\in\mathcal{X}_i,\forall i,$
  $$\Vert\nabla u_i(x_i,y)-\nabla u_i(\hat{x}_i,y)\Vert\leq L_i\Vert x_k-\hat{x}_i\Vert,$$
  where $L_i>0$ is some constant. Further, we have $\forall i,\forall y,z\in\mathcal{X}$
  $$\Vert\nabla u_i(x_i,y)-\nabla u_i(x_i,z)\Vert\leq G_i\Vert y-z\Vert,~\forall x_i\in\mathcal{X}_i.$$
  Define $L_{{\rm max}}=\max_i L_i,G_{{\rm max}}=\max_i G_i$.
\end{enumerate}
\end{definition}

\begin{proposition}\label{thm_innerconverge}
Let $\{X^k\}=\{(X_0^k,\dots,X_M^k)\}$ be the sequence generated by \eqref{u2}, \eqref{u3} for solving optimization problem \eqref{dualproblem2}. Then, it holds
\begin{equation*}
  {\Theta}(X^k)-{\Theta}^*\leq\frac{c_1}{\sigma_1}\frac{1}{k},~\forall k\geq 1,
\end{equation*}
where $\gamma=\frac{1}{2}\min_k\gamma_k$,
\begin{equation*}
  \begin{aligned}
  &R=\max\limits_{X\in\mathcal{X}}\max\limits_{X^*\in\mathcal{X}^*}\left\{\Vert X-X^*\Vert\colon \widehat{\Theta}(X)\leq\widehat{\Theta}(X^0)\right\},\\
  &c_1=\max\{4\sigma_1-2,\widehat{\Theta}(X^0)-\widehat{\Theta}^*,2\},\\
  &\sigma_1=\frac{\gamma}{(M+1)G_{\rm max}^2R^2}.
  \end{aligned}
\end{equation*}
\end{proposition}
\begin{proof}
Denote
\begin{equation*}
\begin{aligned}
&g(X_0,\dots,X_M)=\Vert \sqrt{-\Psi}\mathcal{L}_k(X_0,\dots,X_M)\Vert^2,\\
&h_0(X_0)=\langle X_0,\Omega_k\rangle,\\
&h_i(X_i)=-\langle X_i,T_2(V^\top\otimes V){\rm vec}(Q)\\
&~~~~~~~~~~~+T_2[(V\otimes VF_i)+(VF_i\otimes V)]T_1^\top\Omega_k\rangle,
\end{aligned}
\end{equation*}
for $i=1,\dots,M$. Then, optimization problem (\ref{dualproblem2}) is equivalent to the following standard form:
\begin{equation*}
\begin{aligned}
\min\limits_{X_0,\dots,X_M} &\widehat{\Theta}(X_0,\dots,X_M)=g(X_0,\dots,X_M)+\sum_{i=0}^{M}h_i(X_i)\\
{\rm s.t.}~~~&X_0\times X_1\dots\times X_M\in\widetilde{\Gamma}_+^p\times\widetilde{\Gamma}_+^n\times\dots\times\widetilde{\Gamma}_+^n.
\end{aligned}
\end{equation*}
For $i=0,\dots,M$, let
\begin{equation*}
u_i(X_i,X^k)=g(X_0^k,\dots,X_i,\dots,X_M^k)+\frac{1}{2}\Vert X_i-X_i^k\Vert_{\mathcal{S}_i}^2,
\end{equation*}
with $X^k=(X_0^k,\dots,X_M^k)$. Then, by the update rule \eqref{u1}, \eqref{u2}, \eqref{u3}, sGS-PCD can be regarded as a special case of BSUM, while $\{u_i\}_{i=0}^M$ is a sequence of  strict auxiliary functions of $\widehat{\Theta}$; see Definition \ref{strict_aux}. Hence, this proposition is a direct corollary of Theorem 1 of \cite{b6}.
\end{proof}

For $i=1,\dots, M$, the relative residual error can be defined as
\begin{equation}\label{u4}
\begin{aligned}
{\rm err}_{X_0}^k&=\frac{\Vert X_0^k-\Pi_{\widetilde{\Gamma}_+^p}(X_0^k-L_0X_0^k+H_0)\Vert}{1+\Vert X_0^k\Vert+\Vert L_0X_0^k-H_0 \Vert},\\
{\rm err}_{X_i}^k&=\frac{\Vert X_i^k-\Pi_{\widetilde{\Gamma}_+^n}(X_i^k-L_iX_i^k+H_i)\Vert}{1+\Vert X_i^k\Vert+\Vert L_iX_i^k-H_i \Vert}.
\end{aligned}
\end{equation}
Let
\begin{equation}\label{u5}
{\rm err}^k=\max\{{\rm err}_{X_0}^k,{\rm err}_{X_1}^k,\dots,{\rm err}_{X_M}^k\},
\end{equation}
and the inner layer optimization process will terminate if ${\rm err}^k<\epsilon$. By (\ref{relation}),
\begin{equation}\label{u6}
{\rm svec}(W)=M^{-1}(q_k-\mathcal{L}_k(X_0^k,X_1^k,\dots,X_M^k))
\end{equation}
is the optimal solution to (\ref{subproblem2}). Based on the aforementioned discussion in this section, the iterative algorithm sGS-PCD is summarized as Algorithm \ref{alg1}.
\begin{breakablealgorithm}
    \caption{sGS-PCD}
    \begin{algorithmic}
    \State \textbf{Given} : Initial point $(X_0^0,X_1^0,\dots,X_M^0)$, stopping criterion parameter $\epsilon$
    \State \textbf{Result} : ${\rm vec}(W)$
    \For{$k=0,1,2,...$}

        \State Update $\bar{X}_M^{k+1},\dots,\bar{X}_1^{k+1},X_0^{k+1}$ by (\ref{u1}), (\ref{u2})
        \State Update $X_1^{k+1},\dots,X_M^{k+1}$ by (\ref{u3})
        \State Compute ${\rm err}_{X_0}^{k+1},{\rm err}_{X_1}^{k+1},\dots,{\rm err}_{X_M}^{k+1}$ by (\ref{u4})
        \State Determine ${\rm err}^{k+1}$ by (\ref{u5})
        \If{${\rm err}^{k+1}<\epsilon$}
           \State Compute ${\rm svec}(W)$ by (\ref{u6})\\
            ~~~~~~~~~~\Return ${\rm vec}(W)=D_1{\rm svec}(W)$
        \EndIf
    \EndFor
    \end{algorithmic}
    \label{alg1}
\end{breakablealgorithm}

\subsection{Convergence Analysis}
Based on the previous discussion in this section, the following pseudocode demonstrates how to solve $\ell_1$ relaxation sparse-feedback LQ problem \eqref{problem2} through the primal-dual splitting method \eqref{Luo}. In this subsection, we will perform a convergence analysis and introduce the variational analysis of problem \eqref{problem2}, especially discussing the asymptotic behaviour as $\gamma\to 0$.
The two-timescale algorithm framework for $\ell_1$ relaxation sparse-feedback LQ control is organized as follows.

\begin{breakablealgorithm}
    \caption{$\ell_1$ Relaxation Sparse-Feedback LQ}
    \begin{algorithmic}
    \State \textbf{Given} : Initial point $\widetilde{W}_0\in\Gamma^p,v_0\in\Gamma^p,\widetilde{P}_0\in\mathbb{R}^{mn},w_0\in\mathbb{R}^{mn},\lambda_0\in\mathbb{R}^{mn},\theta_0=1,\beta_0>0,\gamma_0>0$, smoothness parameter $L_f$
    \State \textbf{Result} : $\widetilde{W}_k$, $\widetilde{P}_k$
    \For{$k=0,1,2,...$}
        \State Compute $\alpha_k$ by (\ref{f1})
        \State Compute $\eta_{g,k},\widetilde{y}_k,\widetilde{\eta}_{f,k},\widetilde{v}_k$ by (\ref{f2}), (\ref{f3}), (\ref{f4}), (\ref{f5})
        \State Determine $u_k=(1+\alpha_k)^{-1}(\widetilde{W}_k+\alpha_kv_k)$,\\
        ~~~~~~~~~~~~~~~~~\hspace{0.125em}$d_k=\nabla f_1(u_k)+\mathcal{A}^\top\lambda_k$
        \State Update $v_{k+1}\leftarrow$Algorithm \ref{alg1}
        \State Update $\widetilde{W}_{k+1}=(1+\alpha_k)^{-1}(\widetilde{W}_k+\alpha_kv_{k+1})$
        \State Update $\bar{\lambda}_{k+1}=\lambda_k+\alpha_k/\theta_k(\mathcal{A}v_{k+1}+\mathcal{B}w_k)$
        \State Update $\widetilde{P}_{k+1}$ by (\ref{P*})
        \State Update $w_{k+1}=\widetilde{P}_{k+1}+\alpha_k^{-1}(\widetilde{P}_{k+1}-\widetilde{P}_k)$
        \State Update $\lambda_{k+1}=\lambda_k+\alpha_k/\theta_k(\mathcal{A}v_{k+1}+\mathcal{B}w_{k+1})$
        \State Update $\theta_{k+1},\gamma_{k+1},\beta_{k+1}$ by (\ref{para})
        \If{Stopping Criterion$==$True}\\
           ~~~~~~~~~~\Return $\widetilde{W}_{k+1},\widetilde{P}_{k+1}$
        \EndIf

    \EndFor
    \end{algorithmic}
    \label{alg2}
\end{breakablealgorithm}

\begin{remark}\label{stopping_criteria}
  Denoting
  \begin{equation*}
    \begin{aligned}
      &r_{k+1}=\mathcal{A}\widetilde{W}_{k+1}+\mathcal{B}\widetilde{P}_{k+1},\\
      &s_{k+1}=\rho \mathcal{A}^\top\mathcal{B}(\widetilde{P}_{k+1}-\widetilde{P}_k)
    \end{aligned}
  \end{equation*}
  with constant $\rho>0$, it is recommended that the stopping criterion of Algorithm \ref{alg2} can be selected as
  \begin{equation*}
    \Vert r_k\Vert\leq \epsilon^{(k)}_{\rm pri},~\Vert s_k\Vert\leq \epsilon^{(k)}_{\rm dua},
  \end{equation*}
  where
  \begin{equation*}
    \begin{aligned}
      &\epsilon^{(k)}_{\rm pri}=\sqrt{M+mn}\epsilon_1+\epsilon_2\max\{ \Vert \mathcal{A}\widetilde{W}_k\Vert,\Vert\mathcal{B}\widetilde{P}_k\Vert\},\\
      &\epsilon^{(k)}_{\rm dua}=p\epsilon_1+\epsilon_2\Vert \mathcal{A}^\top\lambda_k\Vert
    \end{aligned}
  \end{equation*}
  with some constants $\epsilon_1, \epsilon_2$ \cite{boyd2004convex}. Moreover, the stopping criterion can be selected in many other ways; see \cite{wohlberg2017admm} for some examples.
\end{remark}

The convergence analysis of Algorithm \ref{alg2} can be derived, as stated in the following proposition.
\begin{proposition}\label{thm_l_1}
   Under the initial condition
  $\gamma_0>0$, $\beta_0>0$
  and the condition
  \begin{equation}\label{f1}
    \Vert \mathcal{B}\Vert^{2}\alpha_{k}^{2}=\beta_{k}\theta_{k},
  \end{equation}
  it holds that
  \begin{equation*}
  \begin{aligned}
    &\|\mathcal{A}\widetilde{W}_k+\mathcal{B}\widetilde{P}_k\|\leq\theta_k C_1,\\
    &|F(\widetilde{W}_k,\widetilde{P}_k)-F^*|\leq\theta_k C_2
  \end{aligned}
  \end{equation*}
  with $F(\widetilde{W},\widetilde{P})=f(\widetilde{W})+h(\widetilde{P})$ and constants $C_1 ,C_2$.
  Above, $\theta_k$ satisfies
  \begin{equation*}
  \theta_k\lesssim\frac{\|\mathcal{B}\|}{\sqrt{\beta_0}k},
  \end{equation*}
  provided that $\beta_{0}\leq\left\|\mathcal{B}\right\|^{2}$.
\end{proposition}

\begin{proof}
Noticing that $f_1(\widetilde{W})$ is a linear function, $\nabla f_1$ is Lipschitz continuous with Lipschitz parameter $L_{f_1}=0$. Then, this proposition is a direct corollary of Theorem 4.1 of \cite{ref2}.
\end{proof}

The following proposition elucidates the asymptotic behavior of optimization problem \eqref{problem2} when the parameter $\gamma_n\to 0$. Therefore, it is possible to achieve a near-optimal centralized controller by selecting a sufficiently small value for $\gamma$ and utilizing Algorithm \ref{alg2}.

\begin{proposition}\label{gammato0}
Let $\gamma_n\to 0$. Assume that $C^\top C\succ 0$, and denote
\begin{align*}
  F^{\gamma_n}(\widetilde{W})&=f(\widetilde{W})+\gamma_n g((V_2\otimes V_1)\widetilde{W})+\delta_{\widehat{\mathcal{A}}\widetilde{W}=0}(\widetilde{W}) ,\\
  F(\widetilde{W})&=f(\widetilde{W})+\delta_{\widehat{\mathcal{A}}\widetilde{W}=0}(\widetilde{W}).
\end{align*}
Then
$\inf F^{\gamma_n}\to\inf F$ hold.
\end{proposition}
\begin{proof}
  Obviously, $F^{\gamma_n}(\widetilde{W})\stackrel{\rm e}{\longrightarrow}F(\widetilde{W})$ (the epigraph convergence; see Definition 7.1 of \cite{ref6}). Based on the fact
\begin{equation*}
\begin{aligned}
\langle {\rm vec}(R),\widetilde{W}\rangle&=\langle R,W\rangle\geq \lambda_{{\rm min}}(R){\rm tr}(W)\\
&\geq \lambda_{{\rm min}}(R)\Vert W\Vert_2,
\end{aligned}
\end{equation*}
$F(\widetilde{W})$ is level-bounded. Additionally, all $F^{\gamma_n}$ are closed convex function, and thus all the sets ${\rm lev}_{\leq\alpha}F^{\gamma_n}$ is connected. Then, $\{F^{\gamma_n}\}_{n\in\mathbb{N}}$ is eventually level-bounded by 7.32 (c) of \cite{ref6}. Hence,
$\inf F^{\gamma_n}\to\inf F$ by Theorem 7.33 of \cite{ref6}.
\end{proof}

\section{Piecewise Quadratic Relaxation Sparse LQ and Acceleration}\label{acc_section}

In Section \ref{section3}, we have so far established the two-timescale optimization framework to obtain a sparse-feedback gain of LQ problem with $\mathcal{O}({1}/{k})$ convergence rate. In this section, we will consider to utilize different relaxation method of $\ell_0$-norm, and then demonstrate that such convex relaxation admits an accelerated convergence rate by primal-dual splitting method.

Let $g_2:\mathbb{R}\to\mathbb{R}$ be
\begin{equation*}
g_2(x)=\left\{
\begin{aligned}
&\frac{1}{2}a_1x^2+b_1x,~{\rm if}~x\leq 0,\\
&\frac{1}{2}a_2x^2+b_2x,~{\rm if}~x> 0,
\end{aligned}
\right.
\end{equation*}
and
$g_{\mathcal{Q}}(P)=\sum_{i,j}w_{ij}g_2(p_{ij})$
called a piecewise quadratic function.
Obviously, by choosing $a_1,a_2>0,b_1<0<b_2$, $g_{\mathcal{Q}}(P)$ becomes a strongly convex function with modulus
$\mu_{g_{\mathcal{Q}}}=\min_{i,j}\left\{w_{ij}\min\{a_1,a_2\}\right\}.$
Remarkably, we aim to use $g_{\mathcal{Q}}$ to approximate the $\ell_0$-norm and the piecewise quadratic relaxation sparse-feedback LQ problem exhibits the following form
\begin{equation}\label{acc_sparse_LQ}
\begin{aligned}
\min_{\widetilde{W},\widetilde{P}}~~ &f(\widetilde{W})+\gamma g_{\mathcal{Q}}(\widetilde{P})\\
{\rm s.t.}~~&\mathcal{A}\widetilde{W}+\mathcal{B}\widetilde{P}=0.
\end{aligned}
\end{equation}
Remarkably, by (\ref{Luo}), $\widetilde{P}_{k+1}$ can be obtained by following Theorem \ref{thm17}, while $\widetilde{W}_{k+1}$ can be obtained by Algorithm \ref{alg1} likewise; thus, the two-timescale algorithm framework for piecewise quadratic relaxation sparse-feedback LQ control is summarized as Algorithm \ref{alg3}.

\begin{theorem}\label{thm17}
  Provided that $a_1>0,a_2>0,b_1<0<b_2$ and letting
  \begin{equation*}
  \hat{P}_{ij;k+1}^*=\left\{
  \begin{aligned}
  &0,~{\rm if}~\Omega_{ij}^k<\frac{\gamma w_{ij} b_2}{\rho_k},\Omega_{ij}^k\geq0,\\
  &0,~{\rm if}~\Omega_{ij}^k>\frac{\gamma w_{ij} b_1}{\rho_k},\Omega_{ij}^k<0,\\
  &\frac{\rho_k\Omega_{ij}^k-\gamma w_{ij} b_2}{\gamma w_{ij} a_2+\rho_k},~{\rm if}~ \Omega_{ij}^k \geq\frac{\gamma w_{ij} b_2}{\rho_k},\Omega_{ij}^k\geq0,\\
  &\frac{\rho_k\Omega_{ij}^k-\gamma w_{ij} b_1}{\gamma w_{ij} a_1+\rho_k},~{\rm if}~ \Omega_{ij}^k \leq\frac{\gamma w_{ij} b_1}{\rho_k},\Omega_{ij}^k<0
  \end{aligned}
  \right.
  \end{equation*}
  with $\Omega^k$ given in \eqref{Omega}, it holds
  \begin{equation}\label{acc_g}
    \widetilde{P}_{k+1}=(\hat{P}_{11;k+1}^*,\hat{P}_{21;k+1}^*,\dots, \hat{P}_{mn;k+1}^*)^\top.
  \end{equation}
\end{theorem}
\begin{proof}
By the definition of $\widetilde{P}_{k+1}$, we have
\begin{align*}
  \widetilde{P}_{k+1}&=\mathbf{prox}_{\tau_{k}\gamma g_{\mathcal{Q}}}^{\mathbb{R}^{mn\times 1}}(\widetilde{P}_{k}-\tau_{k}\mathcal{B}^{\top}\bar{\lambda}_{k+1})\\
  &=\left(\underset{P_{ij}\in\mathbb{R}}{\operatorname*{argmin}} \left\{\gamma w_{ij} g_2(P_{ij})+\frac{\rho_k}{2}\left(P_{ij}-\Omega^k_{ij}\right)^2\right\}\right)_{mn},
\end{align*}
Noticing the fact that for all $x,y\in\mathbb{R}$ and $\alpha\in[0,1]$
$$\vert \alpha x+(1-\alpha)y\vert^2=\alpha |x|^2+(1-\alpha)|y|^2-\alpha(1-\alpha)|x-y|^2,$$
it is obvious that optimization problem
\begin{equation}\label{sub}
\hat{P}_{ij;k+1}^*\in\min_{P_{ij}\in\mathbb{R}} \gamma w_{ij}g_2(P_{ij})+ \frac{\rho_k}{2}\left(x-\Omega^k_{ij}\right)^2
\end{equation}
is strongly convex. Thus, $\hat{P}_{ij;k+1}^*$ exists and is unique.
Denote $\varphi(P_{ij})=\gamma w_{ij} g_2(P_{ij})+ \frac{\rho_k}{2}\left(P_{ij}-\Omega^k_{ij}\right)^2$. When $\Omega_{ij}^k\geq 0$, we can obtain
\begin{equation*}
  \varphi(P_{ij})=\left\{
  \begin{aligned}
  &\left(\frac{\gamma w_{ij} a_1+\rho_k}{2}\right)P_{ij}^2+\left(\gamma w_{ij} b_1-\rho_k\Omega_{ij}^k\right)P_{ij}\\
  &\qquad\qquad\qquad+\frac{\rho_k}{2}{\Omega_{ij}^k}^2,~{\rm if}~P_{ij}\leq 0,\\
  &\left(\frac{\gamma w_{ij} a_2+\rho_k}{2}\right)P_{ij}^2+\left(\gamma w_{ij} b_2-\rho_k\Omega_{ij}^k\right)P_{ij}\\
  &\qquad\qquad\qquad+\frac{\rho_k}{2}{\Omega_{ij}^k}^2,~{\rm if}~P_{ij}> 0.
  \end{aligned}
  \right.
\end{equation*}
Hence, $\varphi(P_{ij})$ monotonously decreases on $(-\infty,0]$. Moreover, if $\Omega_{ij}^k<\frac{\gamma w_{ij} b_2}{\rho_k}$, $\varphi(P_{ij})$ will monotonously increase on $[0,+\infty)$, and if $\Omega_{ij}^k\geq\frac{\gamma w_{ij} b_2}{\rho_k}$, $\varphi(P_{ij})$ will monotonously decrease on $\left[0,\frac{\rho_k\Omega_{ij}^k-\gamma w_{ij} b_2}{\gamma w_{ij} a_2+\rho_k}\right]$, while increase on $\left[\frac{\rho_k\Omega_{ij}^k-\gamma w_{ij} b_2}{\gamma w_{ij} a_2+\rho_k},+\infty\right)$. Consequently, when $\Omega_{ij}^k\geq 0$, we have
\begin{equation*}
  \hat{P}_{ij;k+1}^*=\left\{
  \begin{aligned}
  &0,~~~~{\rm if}~\Omega_{ij}^k<\frac{\gamma w_{ij} b_2}{\rho_k},\\
  &\frac{\rho_k\Omega_{ij}^k-\gamma w_{ij} b_2}{\gamma w_{ij} a_2+\rho_k},~~~~{\rm if}~\Omega_{ij}^k\geq\frac{\gamma w_{ij} b_2}{\rho_k}.
  \end{aligned}
  \right.
\end{equation*}
When $\Omega_{ij}^k< 0$, by the same token, we have
\begin{equation*}
  \hat{P}_{ij;k+1}^*=\left\{
  \begin{aligned}
  &0,~~~~{\rm if}~\Omega_{ij}^k>\frac{\gamma w_{ij} b_1}{\rho_k},\\
  &\frac{\rho_k\Omega_{ij}^k-\gamma w_{ij} b_1}{\gamma w_{ij} a_1+\rho_k},~~~~{\rm if}~\Omega_{ij}^k\leq\frac{\gamma w_{ij} b_1}{\rho_k}.
  \end{aligned}
  \right.
\end{equation*}
Here, we finish the proof.
\end{proof}

\begin{breakablealgorithm}
    \caption{Piecewise Quadratic Relaxation Sparse-Feedback LQ}
    \begin{algorithmic}
    \State \textbf{Given} : Initial point $\widetilde{W}_0\in\Gamma^p,v_0\in\Gamma^p,y_0\in\mathbb{R}^{mn},w_0\in\mathbb{R}^{mn},\lambda_k\in\mathbb{R}^{mn},\theta_0=1,\beta_0=\mu_{g_{\mathcal{Q}}},\gamma_0>0$, smoothness parameter $L_f$
    \State\textbf{Result} : $\widetilde{W}_k$, $\widetilde{P}_k$
    \For{$k=0,1,2,...$}
        \State Compute $\alpha_k$ by (\ref{f1_acc})
        \State Compute $\eta_{g,k},\widetilde{y}_k,\widetilde{\eta}_{f,k},\widetilde{v}_k$ by (\ref{f2}), (\ref{f3}), (\ref{f4}), (\ref{f5})
        \State Determine $u_k=(1+\alpha_k)^{-1}(\widetilde{W}_k+\alpha_kv_k)$,\\
        ~~~~~~~~~~~~~~~~~\hspace{0.25em}$d_k=\nabla f_1(u_k)+\mathcal{A}^\top\lambda_k$
        \State Update $v_{k+1}\leftarrow$Algorithm \ref{alg1}
        \State Update $\widetilde{W}_{k+1}=(1+\alpha_k)^{-1}(\widetilde{W}_k+\alpha_kv_{k+1})$
        \State Update $\bar{\lambda}_{k+1}=\lambda_k+\alpha_k/\theta_k(\mathcal{A}v_{k+1}+\mathcal{B}w_k)$
        \State Update $\widetilde{P}_{k+1}$ by (\ref{acc_g})
        \State Update $\lambda_{k+1}=\lambda_k+\alpha_k/\theta_k(\mathcal{A}v_{k+1}+\mathcal{B}w_{k+1})$
        \State Update $\theta_{k+1},\gamma_{k+1},\beta_{k+1}$ by (\ref{para})
        \If{Stopping Criterion$==$True}\\
           ~~~~~~~~~~\Return $\widetilde{W}_{k+1},\widetilde{P}_{k+1}$
        \EndIf

    \EndFor
    \end{algorithmic}
    \label{alg3}
\end{breakablealgorithm}

Interestingly, the only difference between Algorithm \ref{alg2} and Algorithm \ref{alg3} is the update of $\widetilde{P}_{k+1}$. However, the former holds $\mathcal{O}({1}/{k})$ convergence rate, while the latter holds $\mathcal{O}({1}/{k^2})$ convergence rate. Note that Nesterov GD  is an extension of vanilla GD by introducing an additional sequence and modifying the fundamental framework;  Algorithm \ref{alg2} and Algorithm \ref{alg3} have a completely consistent framework. We have established a direct relationship between the acceleration of solving sparse-feedback LQ problems and the design of strongly convex relaxation problems, and to our best knowledge, this phenomenon has not been reported in the field of sparse optimization.
Concretely, the convergence analysis of Algorithm \ref{alg3} is introduced by the following proposition.

\begin{proposition}\label{thm_acc}
   For all $i,j$, let $w_{ij}\neq 0$ and $a_1>0,a_2>0,b_1<0<b_2$.
   Under the initial condition
  $\gamma_0>0$, $\beta_0=\mu_{g_{\mathcal{Q}}}>0$
  and the condition
  \begin{equation}\label{f1_acc}
    \Vert \mathcal{B}\Vert^{2}\alpha_{k}^{2}=\beta_{k}\theta_{k},
  \end{equation}
  it holds
  \begin{equation*}
  \begin{aligned}
    &\|\mathcal{A}\widetilde{W}_k+\mathcal{B}\widetilde{P}_k\|\leq\theta_k \widetilde{C}_1,\\
    &|F(\widetilde{W}_k,\widetilde{P}_k)-F^*|\leq\theta_k\widetilde{C}_2
  \end{aligned}
  \end{equation*}
  with $F(\widetilde{W},\widetilde{P})=f(\widetilde{W})+h(\widetilde{P})$ and constants $\widetilde{C}_1, \widetilde{C}_2$.
  Above, $\theta_k$ satisfies
  \begin{equation*}
  \theta_k\lesssim\min\left\{\frac{\|\mathcal{B}\|}{\sqrt{\beta_0}k},\frac{\|\mathcal{B}\|^2}{\mu_{g_{\mathcal{Q}}}k^2}\right\},
  \end{equation*}
  provided that $\beta_{0}\leq \left\|\mathcal{B}\right\|^{2}$.
\end{proposition}

\begin{proof}
Noticing that $f_1(\widetilde{W})$ is a linear function, $\nabla f_1$ is Lipschitz continuous with parameter $L_{f_1}=0$. In addition, $g_{\mathcal{Q}}$ is $\mu_{g_{\mathcal{Q}}}$-strongly convex. Then, this proposition is a direct corollary of Theorem 4.1 of \cite{ref2}.
\end{proof}

\section{Directly Optimizing $\ell_0$-penalty}\label{section6}

{
Reviewing the big picture of this paper, we want to study the sparse-feedback LQ problem by solving optimization problem \eqref{sparse_feedback_4}, which can be written by
\begin{equation}\label{warmup_l0}
  \begin{aligned}
    \min_{W\in\mathbb{S}^p}~~&\langle R,W\rangle+\gamma \Vert V_1WV_2^\top\Vert_0\\
    {\rm s.t.}~~&W\in\mathscr{C}.
  \end{aligned}
\end{equation}
In this section, problem \eqref{warmup_l0} will be directly studied without relaxation.
By the following discussions, we show that a series of coordinatewise convex problems can serve as an approximation for problem \eqref{warmup_l0}.
Additionally, the variational properties of problem \eqref{warmup_l0} are of concern.
}

\begin{definition}
Let $f\colon\mathcal{D}\to\mathbb{R}$ be a function where $\mathcal{D}\subseteq\mathbb{R}^m$ is a convex set. The directional derivative of $f$ at point $x$ in direction $d$ is defined by
\begin{equation*}
  f'(x;d)\doteq\underset{\lambda\downarrow 0}{\rm liminf}\frac{f(x+\lambda d)-f(x)}{\lambda}.
\end{equation*}
\end{definition}

\begin{definition}
Let $f\colon\mathcal{D}\to\mathbb{R}$ be a function where $\mathcal{D}\subseteq\mathbb{R}^m$ is a convex set. The point $x$ is a stationary point of $f(\cdot)$ if $f'(x;d)\geq 0$ for all $d$ such that $x+d\in\mathcal{D}$.
\end{definition}

\begin{definition}
$z\in{\rm dom}f\subseteq\mathbb{R}^m$ is the coordinatewise minimizer of $f$ with respective to the coordinates in $\mathbb{R}^{m_1},\dots,\mathbb{R}^{m_n}$, $m_1+\dots+m_n=m$ if $\forall k=1,\dots,n$
\begin{equation*}
  f(z+d_k^0)\geq f(z)~\forall d_k\in\mathbb{R}^{m_k}~{\rm with}~z+d_k^0\in{\rm dom}f,
\end{equation*}
where $d_k^0=(0,\dots,d_k,\dots,0)$.
\end{definition}

\begin{definition}
The function $f\colon \mathbb{R}^m\to\mathbb{R}$ is regular at the point $z\in{\rm dom}f$ with respective to the coordinates $m_1,m_2,\dots,m_n$,$m_1+\dots+m_n=m$ if $f'(z;d)\geq0$ for all $d=(d_1,d_2,\dots,d_n)$ with $f'(z;d_k^0)\geq 0$, where $d_k^0\doteq(0,\dots,d_k,\dots,0)$ and $d_k\in\mathbb{R}^{m_k}$ for all $k$.
\end{definition}

\begin{definition}
The function is lower semi-continuous and is the greatest of all the lower semi-continuous functions $g$ such that $g\leq f$. It is called the lower closure of $f$, denoted by ${\rm cl}f$.
\end{definition}

\begin{definition}
For any sequence $\{f^v\}_{v\in\mathbb{N}}$ of functions on $\mathbb{R}^n$, the lower epi-limit ${\rm e\text{-}liminf}_v f^v$ is the function having as its epigraph the outer limit of the sequence of sets ${\rm epi}f^v$:
$${\rm epi}({\rm e\text{-}liminf}_v f^v)\doteq {\rm limsup}_v({\rm epi}f^v). $$
The upper epi-limit ${\rm e\text{-}limsup}_v f^v$ is the function having as its epigraph the inner limit of the sets ${\rm epi}f^v$:
$${\rm epi}({\rm e\text{-}limsup}_v f^v)\doteq {\rm liminf}_v({\rm epi}f^v). $$
When these two functions coincide, the epi-limits function ${\rm e\text{-}lim}f^v$ is said to exist:
$${\rm e\text{-}lim}f^v\doteq{\rm e\text{-}liminf}_v f^v={\rm e\text{-}limsup}_v f^v.$$
\end{definition}

The analysis in this section is based on the following assumption.
\begin{assumption}
$C^\top C\succ 0$.
\end{assumption}

Denoting $\Omega=\{\widetilde{W}\colon \widehat{\mathcal{A}}\widetilde{W}=0\}$, the optimization problem \eqref{warmup_l0} can be expressed in the unconstrained form by the definition of $\mathscr{C}$:
\begin{align}\label{P}
\min_{\widetilde{W}}~&\langle {\rm vec}(R),\widetilde{W}\rangle+\delta_{\Gamma_+^p}(\widetilde{W})+\delta_{\Gamma_+^n}(\widetilde{\Psi}_1)+\dots+\delta_{\Gamma_+^n}(\widetilde{\Psi}_M)\notag\\
&+\delta_{\Omega}(\widetilde{W})+\gamma \Vert (V_2\otimes V_1)\widetilde{W}\Vert_0.\tag{\textbf{P}}
\end{align}
The discontinuous nature of optimization problem (\ref{P}) makes it extremely challenging to design a suitable algorithm.
However, the following fact holds
\begin{equation*}
  \lim\limits_{\sigma \downarrow 0^+} f_{\sigma}(|y|)=I(y)=\left\{
  \begin{aligned}
  &0,~~y=0,\\
  &1,~~\text{otherwise}
  \end{aligned}
  \right.
\end{equation*}
with $f(y)=1-e^{-y}$ and $f_\sigma (y)\doteq f(y/\sigma)$, which implies that $f_\sigma(\vert\cdot\vert)$ can serve as an approximation of $\Vert \cdot\Vert_0$ for sufficiently small $\sigma$.
For $x=(x_1,\dots,x_n)^\top\in\mathbb{R}^n$, we denote $f_{\sigma}(x)=\sum_{i=1}^{n}f_{\sigma}(x_i)$.

Based on the fact that $\lim\limits_{\sigma\downarrow 0^+}f_{\sigma}(|x|)=\Vert x\Vert_0$, a class of optimization problems are proposed:
\begin{align}\label{P_sigma}
\min_{\widetilde{W}}~&\langle {\rm vec}(R),\widetilde{W}\rangle+\delta_{\Gamma_+^p}(\widetilde{W})+\delta_{\Gamma_+^n}(\widetilde{\Psi}_1)+\dots+\delta_{\Gamma_+^n}(\widetilde{\Psi}_M) \notag\\
&+\delta_{\Omega}(\widetilde{W})+\gamma f_{\sigma}(\vert(V_2\otimes V_1)\widetilde{W}\vert).\tag{$\textbf{P}_\sigma$}
\end{align}
Problem (\ref{P_sigma}) is a nonconvex optimization problem since the folded concave penalty term $f_{\sigma}$ is introduced.

To avoid optimizing nonconvex penalty term directly, \cite{ref4} proposes a convex relaxation framework for solving problem with $f_{\sigma}$ penalty.
Remarkably, by Theorem 4.8 of \cite{ref5}, $f_{\sigma}$ penalty can be expressed as
\begin{equation*}
f_{\sigma}(|x|)=\inf\limits_{y\geq 0}\{y^\top |x|+g_{\sigma}^*(-y)\}.
\end{equation*}
Using the notations
\begin{align*}
  &r_1(\widetilde{W})=\langle {\rm vec}(R),\widetilde{W}\rangle+\delta_{\Gamma_+^p}(\widetilde{W})+\delta_{\Gamma_+^n}(\widetilde{\Psi}_1)\\
  &~~~~~~~~~~~~+\dots+\delta_{\Gamma_+^n}(\widetilde{\Psi}_M)+\delta_{\Omega}(\widetilde{W}),\\
  &r_{2,\sigma}(y)=\gamma g_{\sigma}^*(-y),\\
  &h(\widetilde{W},y)=\gamma y^\top\vert(V_2\otimes V_1)\widetilde{W}\vert,
\end{align*}
problem (\ref{P_sigma}) can be denoted in a compact form:
\begin{equation}\tag{$\textbf{P}'_\sigma$}\label{P''}
\min_{\widetilde{W},y\in\mathbb{R}^{mn}_+}~H_{\sigma}(\widetilde{W},y)=r_1(\widetilde{W})+r_{2,\sigma}(y)+h(\widetilde{W},y).
\end{equation}
It is important to highlight that $r_1,r_2$ are closed and convex, but $h$ is nonconvex and nonsmooth. It seems that, compared to optimization problem (\ref{P_sigma}), (\ref{P''}) remains a nonconvex problem but introduces auxiliary variable $y$, which even makes optimization problem (\ref{P''}) more challenging. Interestingly, it is evident that $H_{\sigma}(\widetilde{W},y)$ is closed convex w.r.t. $\widetilde{W}$ and $y$, respectively. This observation suggests that by employing block coordinate descent (BCD),  (\ref{P''}) can be converted to successive convex optimization problems. Specifically, given $y_{k-1}$, update $\widetilde{W}_k$; then given $\widetilde{W}_k$, update $y_k$, i.e.,
\begin{align}
\label{BP1}&\widetilde{W}_k\in\underset{\widetilde{W}\in\mathbb{R}^{p^2}}{\operatorname*{argmin}}~r_1(\widetilde{W})+h(\widetilde{W},y_{k-1}),\\
\label{BP2}&y_k\in\underset{y\in\mathbb{R}_+^{mn}}{\operatorname*{argmin}}~r_{2,\sigma}(y)+h(\widetilde{W}_k,y).
\end{align}
It is important to underline that both (\ref{BP1}) and (\ref{BP2}) may be classified as convex optimization problems. Our objective is then to solve (\ref{BP1}) and (\ref{BP2}) alternatively in order to approximate the global minimizers of nonconvex problem (\ref{P''}). Before discussing the relation between (\ref{P''}) and  (\ref{BP1}) (\ref{BP2}), we first propose the relation between (\ref{P''}) and (\ref{P}). Denote
\begin{align*}
H_{\sigma}(\widetilde{W})&=r_1(\widetilde{W})+\gamma f_{\sigma}(\vert(V_2\otimes V_1)\widetilde{W}\vert),\\
H(\widetilde{W})&=r_1(\widetilde{W})+\gamma\Vert (V_2\otimes V_1)\widetilde{W}\Vert_0,
\end{align*}
and the following theorem proves that
\begin{equation*}
  \text{(\ref{P''})}\Rightarrow\text{(\ref{P})},~~\text{as}~\sigma\to 0^+.
\end{equation*}
\begin{theorem}\label{thm_app}
  Let $\sigma_i\downarrow 0$, then the following statements hold.
  \begin{enumerate}
    \item $\inf H_{\sigma_i}(\widetilde{W})\to\inf H(\widetilde{W})$.
    \item For $v$ in some index set $N\subseteq \mathbb{N}$, the sets ${\operatorname*{argmin}}H_{\sigma_v}$ are nonempty and form a bounded sequence with
        $\limsup_v({\operatorname*{argmin}}H_{\sigma_v})\subset{\operatorname*{argmin}} H.$
    \item For any choice of $\epsilon_i\downarrow 0$ and $\widetilde{W}_i\in\epsilon_i\text{-}{\operatorname*{argmin}}H_{\sigma_i}$, the sequence $\{\widetilde{W}_i\}_{i\in\mathbb{N}}$ is bounded and such that all its cluster points belong to ${\operatorname*{argmin}}H$.
  \end{enumerate}
\end{theorem}

\begin{proof}
Since for $x=(x_1,\dots,x_n)^\top \in\mathbb{R}^n$
\begin{equation*}
f_{\sigma}(x)=\sum_{i=1}^{n}f_{\sigma}(x_i),
\end{equation*}
we have $f_{\sigma_{i+1}}(x)\geq f_{\sigma_i}(x)$ for every $x\geq 0$. Hence, for every $\widetilde{W}\in\mathbb{R}^{p^2}$, it follows that $H_{\sigma_{i+1}}(\widetilde{W})\geq H_{\sigma_{i}}(\widetilde{W})$, and $\{H_{\sigma_{i}}(\widetilde{W})\}_{i\in\mathbb{N}}$ is nondecreasing. By Proposition 7.4 of \cite{ref6}, ${\operatorname*{e-lim}}_i H_{\sigma_i}$ exists and equals $\sup_i[{\operatorname*{cl}}H_{\sigma_i}]$. Based on the fact  $\lim\limits_{\sigma\downarrow 0^+}f_{\sigma}(|x|)=\Vert x\Vert_0$, it follows that $\sup_i[{\operatorname*{cl}}H_{\sigma_i}](\widetilde{W})=H(\widetilde{W})$. Obviously, for every $\sigma_i$, $H_{\sigma_i}(\widetilde{W})\geq r_1(\widetilde{W})$, and $r_1(\widetilde{W})$ is a coercive function based on the fact
\begin{equation}\label{p1}
\begin{aligned}
\langle {\rm vec}(R),\widetilde{W}\rangle&=\langle R,W\rangle\geq \lambda_{{\rm min}}(R){\rm tr}(W)\\
&\geq \lambda_{{\rm min}}(R)\Vert W\Vert_2
\end{aligned}
\end{equation}
for $W\in\mathbb{S}_+^p$. According to Exercise 7.32 of \cite{ref6}, the sequence $\{H_{\sigma_i}\}_{i\in\mathbb{N}}$ is eventually level-bounded. By noticing that $H_{\sigma_i}$ and $H$ are lower semi-continuous and proper, we finish the proof by Theorem 7.33 of \cite{ref6}.
\end{proof}
Hence, theoretically, one can obtain the minimizers of optimization (\ref{P''}) by successively solving optimization problem (\ref{P_sigma}) with $\sigma\downarrow 0$.
{
As mentioned above, we hope to solve problem \eqref{P''} by solving \eqref{BP1} and \eqref{BP2} alternately, known as BCD method.
However, due to the nonconvex and nonsmooth term $h(\widetilde{W},y)$, direct BCD may fail to converge and cause stable cyclic behavior \cite{ref7}.
It is worth to note that the BCD method can be regarded as a special case of BSUM framework (Algorithm \ref{alg4}); and hence, instead of BCD method, we will design suitable algorithm for solving \eqref{P''} under a more general framework, i.e., BSUM framework.
}

{
The auxiliary function (see Definition \ref{auxiliary}) $u_{2,\sigma}(y,X_k)$ can be selected as
\begin{equation}\label{u_2}
u_{2,\sigma}(y,X_k)=\langle {\rm vec}(R),\widetilde{W}_k\rangle+r_{2,\sigma}(y)+h(\widetilde{W}_k,y)
\end{equation}
with $X_k=(\widetilde{W}_k,y_k)$. Then, the subproblem
\begin{equation*}
  \underset{y\in\mathbb{R}^{mn}_+}{\operatorname*{argmin}}~u_{2,\sigma}(y,X_k)
\end{equation*}
within BSUM framework is equivalent to problem \eqref{BP2}, which is further equivalent to the following problem
\begin{equation}\label{q2}
\min\limits_{y\in\mathbb{R}^{mn}_+} g^*_{\sigma}(-y)+y^\top\vert(V_2\otimes V_1)\widetilde{W}_k\vert.
\end{equation}
}
\begin{lemma}\label{lemma_close}
The global minimizer $y_{k,\sigma}^*$ of optimization problem (\ref{q2}) (equivalently, (\ref{BP2})) is unique, and can be expressed as
\begin{equation}\label{y_k}
y_{k,\sigma}^*=\nabla f_{\sigma}(\vert(V_2\otimes V_1)\widetilde{W}_k\vert).
\end{equation}
\end{lemma}
\begin{proof}
Optimality condition implies that
\begin{equation*}
  0\in-\partial g^*_{\sigma}(-y_{k,\sigma}^*)+\vert(V_2\otimes V_1)\widetilde{W}_k\vert.
\end{equation*}
Since $g^*_{\sigma}$ is a proper, closed and convex function, by Proposition 11.3 of \cite{ref6}, we have
\begin{equation*}
  -y_{k,\sigma}^*\in\partial g_{\sigma}(\vert(V_2\otimes V_1)\widetilde{W}_k\vert).
\end{equation*}
Noticing that $g_{\sigma}=-f_{\sigma}$ is a smooth function on $\mathbb{R}^{mn}_+$, then subgradient operator degenerates to gradient operator, i.e.,
\begin{equation*}
  y_{k,\sigma}^*=\nabla f_{\sigma}(\vert(V_2\otimes V_1)\widetilde{W}_k\vert).
\end{equation*}
The proof is complete.
\end{proof}

We then discuss the selection of $u_{1,\sigma}(\widetilde{W},X_k)$. The proximal term is introduced, i.e.,
\begin{align}\label{u_1}
u_{1,\sigma}(\widetilde{W},X_k;\lambda)=&\langle {\rm vec}(R),\widetilde{W}\rangle+r_{2,\sigma}(y_k)+h(\widetilde{W},y_k)\notag\\
&+\Vert\widetilde{W}-\widetilde{W}_k\Vert^2/(2\lambda),
\end{align}
where $\lambda\in\mathbb{R}$ is a constant.
Now, we direct our attention on solving the subsequent optimization problem rather than (\ref{BP1})
\begin{equation}\label{q3}
\begin{aligned}
 \min\limits_{\widetilde{W}}~ &u_{1,\sigma}(\widetilde{W},X_k;\lambda)\\
 \text{s.t.}~&\widetilde{W}\in\mathcal{X}_1,
\end{aligned}
\end{equation}
with
$$\mathcal{X}_1=\{\widetilde{W}\colon \widetilde{W}\in\Gamma_+^p,\widetilde{\Psi}_i\in\Gamma_+^n,i=1,\dots, M, \widehat{A}\widetilde{W}=0\}$$
which is a closed convex set. First, we claim that $\{H_{\sigma}(\widetilde{W}_k)\}$ is monotonically nonincreasing.

\begin{lemma}
Given an initial feasible point $\widetilde{W}_0$, let $\widetilde{W}_k$ be the iterative sequence generated by Algorithm \ref{alg4} with $u_1,u_2$ given by (\ref{u_1}), (\ref{u_2}). Then
$H_{\sigma}(\widetilde{W}_{k+1})\leq H_{\sigma}(\widetilde{W}_k).$
\end{lemma}
\begin{proof}
Let $F(\widetilde{W},X_k)=\langle {\rm vec}(R),\widetilde{W}\rangle+r_{2,\sigma}(y_k)+h(\widetilde{W},y_k)$. According to the fact that $y_k=\nabla f_{\sigma}(\vert(V_2\otimes V_1)\widetilde{W}_k\vert)$,
$F(\widetilde{W},X_k)$ can be equivalently regarded as $F(\widetilde{W},\widetilde{W}_k)$. Obviously, for all $\widetilde{W}$, we have
$$u_{1,\sigma}(\widetilde{W},X_k;\lambda)\geq F(\widetilde{W},\widetilde{W}_k)\geq H_{\sigma}(\widetilde{W}).$$
Hence, the following inequality holds
\begin{equation*}
  \begin{aligned}
  &H_{\sigma}(\widetilde{W}_{k+1})\leq u_{1,\sigma}(\widetilde{W}_{k+1},X_k;\lambda)\leq u_{1,\sigma}(\widetilde{W}_{k},X_k;\lambda)\\
  &=F(\widetilde{W}_k,\widetilde{W}_k)=H_{\sigma}(\widetilde{W}_k).
  \end{aligned}
\end{equation*}
The proof is completed.
\end{proof}

Obviously, $\{H_{\sigma}(\widetilde{W}_k)\}$ is bounded below by $0$. Hence, by basic mathematic analysis, $H_{\sigma}(\widetilde{W}_k)\downarrow H_{\sigma}^{\infty}$, where $H_{\sigma}^{\infty}$ is a positive constant. An inquiry that arises is whether or not $H_{\sigma}^{\infty}=H_{\sigma}^*\doteq\min H_{\sigma}(\widetilde{W})$. Regretfully, such question remains open, but we have the following result.

\begin{theorem}\label{thm_converge}
  Every cluster point $z=(z_1,z_2)$ of the iteration generated by Algorithm \ref{alg4} with $u_1,u_2$ given by (\ref{u_1}), (\ref{u_2}) is a coordinatewise minimizer of the optimization problem (\ref{P''}). In addition, if $H_{\sigma}(\cdot,\cdot)$ is regular at $z$, then $z$ is a stationary point of (\ref{P''}).
\end{theorem}
\begin{proof}
For a feasible $X_0=(\widetilde{W}_0,y_0)$, by (\ref{p1}), the sublevel set
$$\mathcal{X}^0\doteq\{X=(\widetilde{W},y)\colon H_{\sigma}(X)\leq H_{\sigma}(X_0)\}$$
is compact. By Lemma \ref{lemma_close},  optimization problem (\ref{BP2}) has a unique solution for any point $X_{r-1}\in\mathcal{X}$. Then, the proof is completed by Theorem 2 of \cite{ref9}.
\end{proof}

There exists a significant gap between coordinatewise minimizer and stationary point (or local minimizer). When the objective function is regular at coordinatewise minimizer, the coordinatewise minimum becomes a local minimum by the above theorem. In Lemma 3.1 of \cite{ref10}, there is a few discussion about regularity property of a Gateaux-differentiable function, but $H_{\sigma}(\widetilde{W},y)$ is generally not Gateaux-differentiable in this paper. Furthermore, if we assume that $H_{\sigma}^{\infty}=H_{\sigma}^*$ holds, we can obtain the following theorem.

\begin{theorem}\label{thm_opt}
  If $H_{\sigma}^{\infty}=H_{\sigma}^*$, then any cluster point of $\{\widetilde{W}_k\}$
  belongs to ${\rm argmin}H_{\sigma}(\widetilde{W})$.
\end{theorem}
\begin{proof}
Since $f_{\sigma}(\vert(V_2\otimes V_1)\widetilde{W}\vert)\geq 0$, it is obvious that, for all $\alpha\in\mathbb{R}$, ${\rm lev}_{\leq\alpha}H_{\sigma}(\widetilde{W})$ is compact. Then the set
$\{\widetilde{W}\colon H_{\sigma}(\widetilde{W})\leq H_{\sigma}(\widetilde{W_0})\}$
is compact. Hence, by $H_{\sigma}(\widetilde{W}_k)\downarrow H_{\sigma}^{\infty}$, $\{\widetilde{W}_k\}_{k\in\mathbb{N}}$ is compact. Thus, there exists a subsequence $\{\widetilde{W}_{k_n}\}_{n\in\mathbb{N}}$ of $\{\widetilde{W}_k\}$ that converges to $\widetilde{W}_{\infty}$. Based on the fact that $H_{\sigma}(\widetilde{W})$ is lower semi-continuous, we have
$$H_{\sigma}^*=H_{\sigma}^{\infty}=\underset{k\to\infty}{\rm liminf}H_{\sigma}(\widetilde{W}_k)\geq H_{\sigma}(\widetilde{W}_{\infty})\geq H_{\sigma}^*.$$
The proof is completed.
\end{proof}
Our focus is now directed towards solving subproblem (\ref{q3}).
By introducing augmented variable $\widetilde{P}=(V_2\otimes V_1)\widetilde{W}$, \eqref{q3} becomes
\begin{equation}\label{q6}
  \begin{aligned}
  \min_{\widetilde{W},\widetilde{P}}~~&f_1(\widetilde{W};\lambda)+f_2(\widetilde{W})+\xi(\widetilde{P})\\
  {\rm s.t.}~~&\mathcal{A}\widetilde{W}+\mathcal{B}\widetilde{P}=0
  \end{aligned}
\end{equation}
with
\begin{align*}
&f_{1}(\widetilde{W};\lambda)=\langle {\rm vec}(R),\widetilde{W}\rangle+ \frac{1}{2\lambda}\Vert\widetilde{W}-\widetilde{W}_k\Vert^2,\\
&f_2(\widetilde{W})=\delta_{\Gamma_+^p}(\widetilde{W})+\delta_{\Gamma_+^n}(\widetilde{\Psi}_1)+\dots+\delta_{\Gamma_+^n}(\widetilde{\Psi}_M),\\
&\xi(\widetilde{P})=\gamma y_k^\top\vert\widetilde{P}\vert,\\
&f(\widetilde{W})=f_1(\widetilde{W})+f_2(\widetilde{W}).
\end{align*}
Interestingly, $\xi(\widetilde{P})$ is nothing but a weighted-$\ell_1$-norm with weight values given by $y_k$. Hence, compared to optimization problem (\ref{problem2}), the only difference is the smooth and strongly convex term $\frac{1}{2\lambda}\Vert\widetilde{W}-\widetilde{W}_k\Vert^2$ in $f_1(\widetilde{W};\lambda)$. Thus, we can utilize optimization scheme (\ref{Luo}) by slightly modifying Algorithm \ref{alg2} (just change $w_{ij}$ to $y_k^{(ij)}$, where $y_k^{(ij)}$ is the $(i,j)$-th element of $y_k$).

\begin{breakablealgorithm}
  \caption{$\ell_0$ Sparse-Feedback LQ}
  \begin{algorithmic}
  \State \textbf{Given} : Feasible initial point $X_0=(\widetilde{W}_0,y_0)\in\mathcal{X}_1\times \mathbb{R}_+^{mn},\lambda>0,\gamma>0,\sigma_0\leq1,\alpha\in(0,1)$
  \State \textbf{Result}: $X_k=(\widetilde{W}_k,y_k)$
  \For{$i=1,2,3\dots$}
  \State Set $\sigma_i=\alpha\sigma_{i-1}$
  \State Set $k=0$
  \For{$k=0,1,2,...$}
      \State Let $k=k+1,i=(k\mod 2)+1$
      \If{$i==1$}
      \State $\widetilde{W}_k\leftarrow$ Algorithm \ref{alg1}
      \State Set $y_k=y_{k-1},X_k=(\widetilde{W}_k,y_k)$
      \EndIf
      \If{$i==0$}
      \State Update $y_k=\nabla f_{\sigma_i}\left(\left\vert(V_2\otimes V_1)\widetilde{W}_k\right\vert\right)$
      \State Set $\widetilde{W}_k=\widetilde{W}_{k-1},X_k=(\widetilde{W}_k,y_k)$
     \EndIf
      \If{Stopping Criterion$==$True}
        \State \textbf{Break}
      \EndIf
  \EndFor
  \State Set $X_0=X_k$
  \EndFor
  \end{algorithmic}
  \label{alg5}
\end{breakablealgorithm}

\begin{remark}
In fact, if $u_{1,\sigma}(\widetilde{W},X_k)$ is selected as
\begin{equation*}
u_{1,\sigma}(\widetilde{W},X_k)=\langle {\rm vec}(R),\widetilde{W}\rangle+r_{2,\sigma}(y_k)+h(\widetilde{W},y_k),
\end{equation*}
the iterative sequence of the algorithm BSUM can also converge to the coordinatewise minimizers of $H_{\sigma}(\widetilde{W},y)$. However, we introduce the proximal term $\frac{1}{2\lambda}\Vert\widetilde{W}-\widetilde{W}_k\Vert^2$ into $u_{1,\sigma}$, which appears to be meaningless or potentially resulting in a decelerated convergence of the BSUM algorithm. Empirically, the proximal term  makes the objective of (\ref{q6}) be strongly convex and accelerate the solving for subproblem (\ref{q6}). Therefore, a trade-off exists between the convergence rate of the algorithm BSUM for optimization problem (\ref{P''}) and the convergence rate of (\ref{Luo}) for subproblem (\ref{q6}).
\end{remark}

Building upon the previous discussions, we have introduced a method with significant theoretical complexity. However, it should be noted that there is no theoretical evidence to support its superiority over Algorithm \ref{alg2} (or Algorithm \ref{alg3}). Therefore, this section primarily focuses on theoretical aspects, particularly the variational properties. For practical purposes, especially when seeking sparse controllers, we recommend the use of Algorithm \ref{alg2} (or Algorithm \ref{alg3}).

\section{Connection between Sparse Feedback and Distributed Controller}\label{section_link}

Distributed LQ problems, with the advantages of scalability, robustness, flexibility, and efficiency, have been widely studied in the existing literature; see \cite{i9,i12,b6,i10,i11} and references therein for examples. In this section, we will establish the relationship between sparse-feedback LQ and distributed controller and demonstrate that the sparse-feedback LQ problem may degenerate to the distributed LQ problem to some extents.

Generally, a sparse (or group-sparse) feedback gain $K$ can lead to a multi-agent distributed control system. For example, given a group-sparse feedback gain
  \begin{equation*}
    K=\begin{bmatrix}
      K_{11} & \cdots & K_{1q}\\
      \vdots & \ddots & \vdots\\
      K_{p1} & \cdots & K_{pq}
    \end{bmatrix}\in\mathbb{R}^{m\times n}
  \end{equation*}
  with $K_{ij}\in\mathbb{R}^{m_i\times n_j}$ and $\sum_{i=1}^{p} m_i=m,\sum_{j=1}^{q}n_j=n$, it can induce a $p$-agents distributed system.
  Specifically, the global control action is composed of local control actions: $u(t)=[u_1(t)^\top,\dots,u_p(t)^\top]^\top$, where $u_i(t)\in\mathbb{R}^{m_i}$ is the control input of
  agent $i$. At time $t$, agent $i$ directly observes a partial state $x_{\mathcal{I}_i}(t)$, and $\mathcal{I}_i$ is given by
  \begin{equation*}
    \mathcal{I}_i=\bigcup_{j=1,\dots,q,~K_{ij}\neq 0}\left\{n_{j-1}+1,\dots,n_j\right\}
  \end{equation*}
  with $n_0=0$.
  Here, $\mathcal{I}_i$ is a fixed subset of $[n]\doteq\{1,\dots, n\}$ and $x_{\mathcal{I}_i}(t)$ denotes the subvector of $x(t)$ with indices in $\mathcal{I}_i$.
  That is why in this paper we investigate the distributed LQ problem by studying optimization problem
  \begin{equation}\label{why_sparse_opt}
    \begin{aligned}
      \min_{W\in\mathbb{S}^p}~~&\langle R,W\rangle +\gamma s(V_1WV_2^\top)\\
      {\rm s.t.}~~&W\in\mathscr{C},
    \end{aligned}
\end{equation}
where $s(\cdot)$ is the approximation of $\ell_0$-norm.
It is important to note that the $p$-agents distributed systems in \cite{ref1,i12,i9} satisfy
\begin{equation*}
  \mathcal{I}_i \cap \mathcal{I}_j =\emptyset,~~\forall i,j\in[p],i\neq j.
\end{equation*}
However, $\mathcal{I}_i$ and $\mathcal{I}_j$ may overlap in our setting, and hence the considered problem of this paper is a  generalization of those of \cite{ref1,i12,i9}.

In addition, in practical applications of distributed system, two agents are often unable to communicate due to long distances or interference, i.e., the feedback gain
\begin{equation*}
  K\in\mathcal{K}\doteq\{K\in\mathbb{R}^{m\times n}\colon K_{ij}=0 \text{ for all } (i,j)\in\mathcal{U}\}
\end{equation*}
with given subset $\mathcal{U}\subseteq [m]\times [n]$. It is worth noting that our sparse-feedback framework is able to cover such situation by modifying problem \eqref{sparse_feedback_4} as
\begin{equation}\label{fixed}
  \begin{aligned}
  \min_{W,P}~~&\langle R,W\rangle\\
  {\rm s.t.}~~&\mathcal{G}(W)\in\mathcal{K},\\
  &[V_1WV_2^\top]_{ij}=0,~(i,j)\in\mathcal{U}\subseteq [m]\times [n],
  \end{aligned}
\end{equation}
with $[\cdot]_{ij}$ the $(i,j)$-th element of a matrix. By solving problem \eqref{fixed}, the controller with fixed communication topology can be obtained, and distributed LQ problems with fixed communication topology have been widely studied from different viewpoints \cite{i10,i11,furieri2020sparsity,rotkowitz2005characterization}. However, our paper is of independent interests for problems with  fixed communication topology. Specifically, this paper provides an optimal-guarantee w.r.t. \eqref{fixed} rather than heuristic method \cite{i10} or stationary-point-guarantee \cite{i11}. Additionally, the proposed problem \eqref{fixed} is finite-dimensional, in comparison to the infinite-dimensional optimization  \cite{furieri2020sparsity,rotkowitz2005characterization}. Moreover, if we only hope that agent $i_k$ cannot observe the state $x_{j_k}(t)$ ($k=1,\dots,p$) directly rather than giving a fixed communication topology, we can modify problem \eqref{fixed} as
\begin{equation}\label{fixed2}
  \begin{aligned}
  \min_{W,P}~~&\langle R,W\rangle+\gamma s(P)\\
  {\rm s.t.}~~&\mathcal{G}(W)\in\mathcal{K},\\
  &V_1WV_2^\top-P=0,\\
  &[V_1WV_2^\top]_{i_kj_k}=0,~k=1,\dots,p\\
  \end{aligned}
\end{equation}
with $(i_k,j_k)\in[m]\times [n]$ for $k=1,\dots,p$.
To the best of the authors' knowledge, such special formulation, i.e., the design of a sparse-feedback gain $K$ with forcing $[K]_{i_kj_k}=0$ for $k=1,\dots,p$, has not been reported in the existing literature.
It is worthwhile to underline that problem \eqref{fixed2} can be solved analogously by the proposed algorithms in Section \ref{section3}.

Therefore, although the focus of this paper is on the sparse-feedback LQ problem, the methodologies presented also make contributions to the distributed LQ problem.

\section{Numerical Examples}
In this section, a few numerical examples are represented that illustrate the theoretical results of this paper, {where the noise $w$ (see (\ref{expsys})) is characterized by a impulse disturbance vector.}

\textbf{Example 1: }Consider $x=[x_1,x_2,x_3]^\top$ and a linear system
\begin{equation}\label{expsys}
\begin{aligned}
\dot{x}&-Ax+B_2u+B_1w,\\
z&=Cx+Du,\\
u&=-Kx,
\end{aligned}
\end{equation}
where
\begin{equation*}
  \begin{aligned}
  &A=\begin{bmatrix}
  0.2220 & 0.9186 & 0.7659\\
  0.8707 & 0.4884 & 0.5184\\
  0.2067 & 0.6117 & 0.2968
  \end{bmatrix},B_1=I_3, C=\begin{bmatrix}
  1 & 0 & 0\\
  0 & 0 & 0\\
  0 & 0 & 0
  \end{bmatrix},\\
  &B_2=\begin{bmatrix}
  0.9315 & 0.7939\\
  0.9722 & 0.1061\\
  0.5317 & 0.7750
  \end{bmatrix}, D=\begin{bmatrix}
  0 & 0\\
  1 & 0\\
  0 & 1
  \end{bmatrix}.
  \end{aligned}
\end{equation*}
By Algorithm \ref{alg2}, we can obtain the relation between $\gamma$ and $N$ (the number of $0$ in feedback gain $K$; see Fig. \ref{fig2}).

  \begin{table}[]
    \centering
    \begin{tabular}{|c|c|c|c|c|c|c|}
    \hline
     & $\gamma=0$ & $\gamma=1$ & $\gamma=5$ &$\gamma=10$ & $\gamma=20$ & $\gamma=50$ \\ \hline
    $J(K)$ & 1.92 & 3.51 & 5.78 & 6.56 &7.20 &7.91 \\ \hline
    $N$ & 0&1 &2 & 3&3&3\\   \hline
    \end{tabular}
    \caption{Relation between $\gamma$ and $J(K)$}
    \label{tabel1}
  \end{table}

\begin{figure}[htbp]
\centerline{\includegraphics[width=0.30\textwidth]{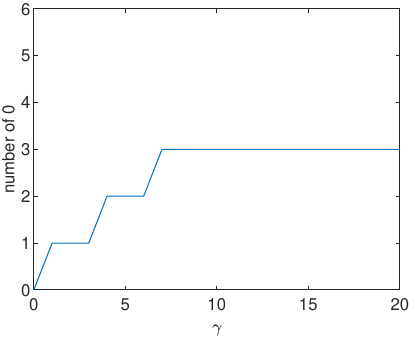}}
\caption{Relation between $\gamma$ and $N$}
\label{fig2}
\end{figure}
{
Remarkably, if $(A,B_2)$ is known, we have $\langle R,W\rangle=J(K)$ rather than an upper bound of $J(K)$; see Remark \ref{explain_ass12} above.
The TABLE \ref{tabel1} illustrates the relation between $\gamma$ and LQ cost $J(K)$, and we can observe that $J(K)$ increases with the increase of $\gamma$.
In particular, when $\gamma=0$, the problem studied in this example degenerates to classic centralized LQ problem, which can be solved by letting $\gamma\to 0$ and utilizing Algorithm \ref{alg2} (or Algorithm \ref{alg3}; see Proposition \ref{gammato0}) or solving Algebraic Riccati Equation.
As illustrated in Fig. \ref{fig2} and TABLE \ref{tabel1}, the parameter $\gamma$ serves as the balance between LQ cost $J(K)$ and the level of sparsity of feedback gain $K$.
Additionally, the LQ cost changes very little and the level of sparsity $N$ remains unchange when $\gamma$ is large enough; which implies if we only need a sufficiently sparse-feedback gain $K$,
we can simply let $\gamma$ sufficiently large.
}

Specifically, when $\gamma=10$, we can obtain the following result (see Fig. \ref{fig3}).
{
Denoting
\begin{equation*}
  W_k=\mathrm{vec}^{-1}(\widetilde{W}_k)=\begin{bmatrix}
    W_{1,k} & W_{2,k}\\
    W_{2,k}^\top &W_{3,k},
  \end{bmatrix},
\end{equation*}
}\noindent
{
where $\{\widetilde{W}_k\}$ is the iteration sequence of Algorithm \ref{alg2}, we have that
\begin{equation*}
  K_k=W_{2,k}^\top W_{1,k}^{-1}=\begin{bmatrix}
    K_{11,k} & K_{12,k} &  K_{13,k}\\
    K_{21,k} & K_{22,k} &  K_{23,k}
    \end{bmatrix}
\end{equation*}
converges to the stabilizing feedback gain by Theorem \ref{thm_para} and Corollary \ref{thm_l_1},
and indeed the colored graphs in Fig. \ref{fig3} indicate the convergence of $K_{ij,k}$ for $i=1,2$, $j=1,2,3$.
}
\begin{figure}[htbp]
\centerline{\includegraphics[width=0.3\textwidth]{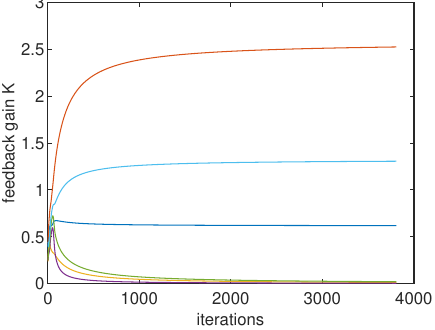}}
\caption{Feasible gain with $\gamma=10$}
\label{fig3}
\end{figure}

The optimal sparse-feedback gain $K$ is given by
\begin{equation*}
  K=\begin{bmatrix}
      0.6192  &  2.5269 &   0\\
    0  &  0  &  1.3068

    \end{bmatrix}.
\end{equation*}
The responses of all the state variables are illustrated in Fig. \ref{fig4}, and it can be seen that the stability of system is guaranteed.
\begin{figure}[htbp]
\centerline{\includegraphics[width=0.3\textwidth]{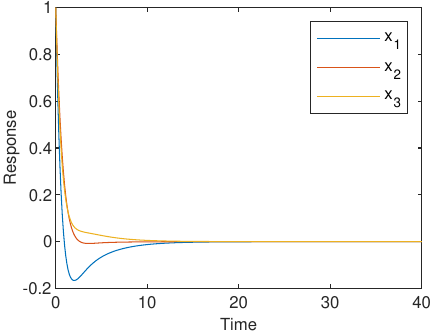}}
\caption{System response}
\label{fig4}
\end{figure}

\textbf{Example 2: }Consider the linear system (\ref{expsys}) with
\begin{equation*}
  \begin{aligned}
  &A=\begin{bmatrix}
  0.3079 & 0.1879 & 0.1797 & 0.2935 & 0.6537\\
  0.5194 & 0.2695 & 0.5388 & 0.9624 & 0.5366\\
  0.7683 & 0.4962 & 0.2828 & 0.9132 & 0.9957\\
  0.7892 & 0.7391 & 0.7609 & 0.5682 & 0.1420\\
  0.8706 & 0.1950 & 0.2697 & 0.4855 & 0.9753
  \end{bmatrix},\\
  &B_2=\begin{bmatrix}
  0.6196 & 0.6414\\
  0.7205 & 0.9233\\
  0.2951 & 0.8887\\
  0.6001 & 0.6447\\
  0.7506 & 0.2956
  \end{bmatrix},\\
  & B_1=I_5,C=\begin{bmatrix}
  1 & 0 & 0 & 0 & 0\\
  0 & 1 & 0 & 0 & 0\\
  0 & 0 & 0 & 0 & 0\\
  0 & 0 & 0 & 0 & 0\\
  0 & 0 & 0 & 0 & 0
  \end{bmatrix},D=\begin{bmatrix}
  0 & 0\\
  0 & 0\\
  1 & 0\\
  0 & 1\\
  0 & 0
  \end{bmatrix}.
  \end{aligned}
\end{equation*}
Letting $\gamma=10$, the optimal sparse-feedback gain $K$ is given by
\begin{equation*}
  K=\begin{bmatrix}
  1.3355 & 0 & 0 &  0 & 5.5233\\
    0.8946 &  1.8952  &  1.0401  & 3.8119 &  0
  \end{bmatrix}.
\end{equation*}
The responses of all the state variables are illustrated in Fig. \ref{fig5}, and the stability of system is guaranteed.
\begin{figure}[htbp]
\centerline{\includegraphics[width=0.3\textwidth]{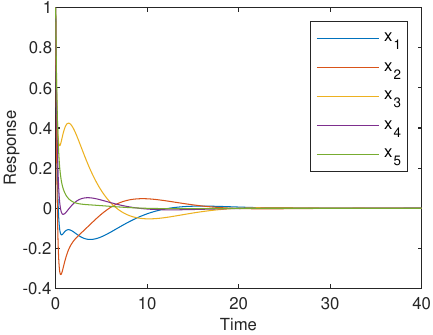}}
\caption{System response}
\label{fig5}
\end{figure}

\textbf{Example 3: }Consider the linear system (\ref{expsys}) with
\begin{equation*}
  \begin{aligned}
  &A=\begin{bmatrix}
  0 & 1 & 0\\
  0 & 0 & 1\\
  0 & 0 & 0
  \end{bmatrix},B_1=I_3,B_2=\begin{bmatrix}
  0.9315 & 0.7939\\
  0.9722 & 0.1061\\
  0.5317 & 0.7750
  \end{bmatrix},\\
  &C=\begin{bmatrix}
  1 & 0 & 0\\
  0 & 0 & 0\\
  0 & 0 & 0
  \end{bmatrix},D=\begin{bmatrix}
  0 & 0\\
  1 & 0\\
  0 & 1
  \end{bmatrix}.
  \end{aligned}
\end{equation*}
We utilize  Algorithm \ref{alg2} and Algorithm \ref{alg3} with  $\gamma=10$, respectively, and the results are shown below (see Fig. \ref{fig6} and Fig. \ref{fig7}).
\begin{figure}[htbp]
\centerline{\includegraphics[width=0.3\textwidth]{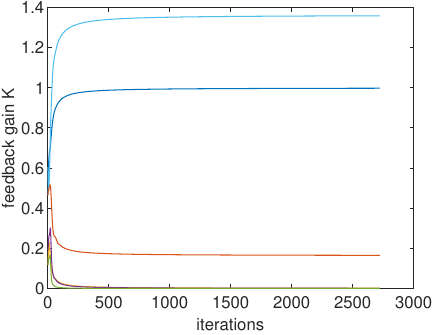}}
\caption{Convergence of Algorithm \ref{alg2}}
\label{fig6}
\end{figure}
\begin{figure}[htbp]
\centerline{\includegraphics[width=0.3\textwidth]{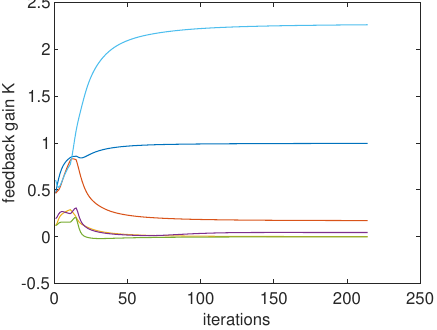}}
\caption{Convergence of Algorithm \ref{alg3}}
\label{fig7}
\end{figure}

{
By using Algorithms \ref{alg2} and \ref{alg3}, denote the obtained feedback gain as $K_1, K_2$, respectively; we have $J(K_1)=8.81$ and $J(K_2)=9.39$.
Moreover, it holds $\Vert K_1\Vert_0=3$ and $\Vert K_2\Vert_0=4$.
Generally, the obtained LQ cost $J(K_1)$ is slightly smaller than the one by Algorithm \ref{alg3}; this is mainly caused by $g_{\mathcal{Q}}(x)\gg g(x)$  when $\Vert x\Vert$ is large.
Additionally, the obtained feedback gains exhibit nearly the same level of sparsity by using Algorithms \ref{alg2} and \ref{alg3},
however, the required iteration numbers could be really different.
As shown in Fig. \ref{fig6} and \ref{fig7}, we can observe that Algorithm \ref{alg2} converges within 800 steps, while Algorithm \ref{alg3} converges within only 80 steps.
That is because the Algorithm \ref{alg2} exhibits $\mathcal{O}(1/k)$ convergence rate, while Algorithm \ref{alg3} exhibits $\mathcal{O}(1/k^2)$ convergence rate.
}

\section{Conclusion}
This work investigates the $\mathcal{H}_2$-guaranteed cost sparse-feedback control problem under convex parameterization and convex-bounded uncertainty, specifically focusing on the inclusion of a sparsity constraint on the feedback gain.
The formulation of approximate separable constraint optimization problems is presented, and a convex optimization framework is suggested to achieve the global optimizer for these problems.
In the near future, an examination will be conducted on sparse-feedback LQ problems with nonconvex optimization perspective.

\bibliographystyle{plain}
\bibliography{tac_7.9}
\end{document}